\documentclass[reqno,11pt]{amsart}
\usepackage{amsmath,amssymb,amsthm,graphicx,a4wide,enumerate,url}
\usepackage[small,bf]{caption} 
\usepackage{amsmath}
\usepackage{color}
\usepackage{lipsum}
\usepackage{amsfonts}
\usepackage{graphicx}
\usepackage{epstopdf}
\usepackage{algorithmic}
\usepackage[outercaption]{sidecap}
\usepackage{amsopn}
\usepackage{enumerate}
\usepackage{array}
\usepackage{makecell}
\usepackage{adjustbox}
\usepackage{calc}
\usepackage{accents}
\usepackage{cleveref}

\newtheorem{theorem}{Theorem}[section]
\newtheorem{remark}[theorem]{Remark}

\newtheorem{lemma}[theorem]{Lemma}

\newtheorem{definition}[theorem]{Definition}

\renewcommand{\O}{{\mathcal O}}

\newcommand{\SCondAlb}{S}

\begin{document}
\parskip.9ex

\title[Design of DIRK Schemes with High Weak Stage Order]
{Design of DIRK Schemes with High Weak Stage Order}

\author[A. Biswas]{Abhijit Biswas}
\address[Abhijit Biswas]
{Computer, Electrical, and Mathematical Sciences \& Engineering Division \\ 
King Abdullah University of Science and Technology \\ Thuwal 23955 \\ Saudi Arabia} 
\email{abhijit.biswas@kaust.edu.sa}
\urladdr{https://math.temple.edu/\~{}tug14809}

\author[D. Ketcheson]{David Ketcheson}
\address[David Ketcheson]
{Computer, Electrical, and Mathematical Sciences \& Engineering Division \\ 
King Abdullah University of Science and Technology \\ Thuwal 23955 \\ Saudi Arabia} 
\email{david.ketcheson@kaust.edu.sa}
\urladdr{https://www.davidketcheson.info}

\author[B. Seibold]{Benjamin Seibold}
\address[Benjamin Seibold]
{Department of Mathematics \\ Temple University \\
1805 North Broad Street \\ Philadelphia, PA 19122}
\email{seibold@temple.edu}
\urladdr{http://www.math.temple.edu/\~{}seibold}

\author[D. Shirokoff]{David Shirokoff}
\address[David Shirokoff]
{Corresponding author, Department of Mathematical Sciences \\ New Jersey Institute of Technology \\ University Heights \\ Newark, NJ 07102}
\email{david.g.shirokoff@njit.edu}
\urladdr{https://web.njit.edu/\~{}shirokof}


\begin{abstract}
Runge-Kutta (RK) methods may exhibit order reduction when applied to certain stiff problems. While fully implicit RK schemes exist that avoid order reduction via high-stage order, DIRK (diagonally implicit Runge-Kutta) schemes are practically important due to their structural simplicity; however, these cannot possess high stage order. The concept of weak stage order (WSO) can also overcome order reduction, and it is compatible with the DIRK structure. DIRK schemes of WSO up to $3$ have been proposed in the past, however, based on a simplified framework that cannot be extended beyond WSO 3. In this work a general theory of WSO is employed to overcome the prior WSO barrier and to construct practically useful high-order DIRK schemes with WSO $4$ and above. The resulting DIRK schemes are stiffly accurate, L-stable, have optimized error coefficients, and are demonstrated to perform well on a portfolio of relevant ODE and PDE test problems.
\end{abstract}

\maketitle
\renewcommand{\O}{{\mathcal O}}
\newcommand{\dbtilde}[1]{\accentset{\approx}{#1}}
\newcommand{\vardbtilde}[1]{\tilde{\raisebox{0pt}[0.85\height]{$\tilde{#1}$}}}

\newcommand{\Err}{\mathcal{I}(\zeta)}  
\newcommand{\myremarkend}{$\spadesuit$} 
\newcommand{\lbmap}{\left[} 
\newcommand{\rbmap}{\right]} 
\newcommand{\WSOset}{\mathbb{W}_q}
\newcommand{\OCset}{\mathbb{V}_{p}}
\newcommand{\numc}{n_c}
\newcommand{\minPolc}{P_{\rm C}}
\newcommand{\qso}{\tilde{q}}
\newcommand{\RzCoeff}{\beta}
\newcommand{\Sredr}{\tilde{r}}
\newcommand{\BigO}{{\mathcal{O}}}
\newcommand{\Dt}{\Delta t}
\newcommand{\vecpower}[2]{\vec{#1}^{\,#2}}
\newcommand{\kgen}{m}
\newcommand{\basismat}{L}
\newcommand{\Umat}{U}
\newcommand{\Vmat}{V}
\newcommand{\Wmat}{W}
\newcommand{\bsmall}{\hat{\vec{b}}}
\newcommand{\Ksmall}{\hatK}
\newcommand{\kibitz}[2]{\textcolor{#1}{#2}}
\newcommand{\ben}[1] {\kibitz{magenta}{[BS says: #1]}}
\newcommand{\abhi}[1] {\kibitz{blue}{[AB says: #1]}}
\newcommand{\dave}[1] {\kibitz{red}{[DS says: #1]}}
\newcommand{\david}[1] {\kibitz{purple}{[DK says: #1]}}
\newcommand{\red}[1]{\kibitz{red}{#1}}

\section{Introduction}
\label{sec:introduction}
This paper focuses on Runge-Kutta (RK) methods for initial value problems
\begin{equation}\label{eq:IVP}
         u^{\prime}(t)  = f(t,u(t)), \  u(0) = u_{0}; \ u\in \mathbb{R}^{m}, \ f : \mathbb{R}\times\mathbb{R}^m \to \mathbb{R}^m\;.
\end{equation}
Let $u_n$ and $u_{n+1}$ denote the numerical approximations to the true solution at times $t_n$ and $t_{n+1} = t_n+\Delta t$, respectively, where $\Delta t$ is the time step size. One step of the RK method reads as
\begin{equation}\label{RK_methods_1}
    u_{n+1} = u_{n}+\Delta t \sum_{j = 1}^{s} b_j f(t_n+c_j \Delta t, g_j)\;,
\end{equation}
via the stage approximations
\begin{equation}\label{RK_methods_2}
    g_i = u_n+\Delta t \sum_{j = 1}^{s} a_{ij} f(t_n+c_j \Delta t, g_j), \quad i = 1,2, \ldots,s\;.
\end{equation}
The parameters $A = (a_{ij})_{ij} \in \mathbb{R}^{s \times s}$ and $\vec{b} = (b_1,\dots,b_s)^T, \vec{c} = (c_1,\dots,c_s)^T \in \mathbb{R}^{s}$ that define the $s$-stage RK scheme are displayed via the Butcher tableau
\begin{equation*} 
\renewcommand\arraystretch{1.2}
\begin{array}
{c|c}
\vec{c} & A\\
\hline
& \vecpower{b}{T}
\end{array}\;.
\end{equation*}
Throughout this work, we assume that the abscissas vector $\vec{c}$ is related to $A$ via
\begin{equation} \label{c-assumption}
\vec{c} = A\vec{e}\;,
\end{equation}
where $\vec{e}\in \mathbb{R}^{s}$ is the vector of ones.
Schemes for which $A$ is lower-triangular are called \emph{diagonally implicit Runge-Kutta} (DIRK) methods. Because the DIRK stage equations can be solved in sequence (whereas a fully-implicit RK method requires simultaneous solution of all stages), these methods are of particular practical interest due to their implementation-friendly structure and cost efficiency.

One major drawback of RK methods is that they may exhibit order reduction \cite{biswas2021structure,burrage1990order,Calvo2002,CarpenterGottliebAbarbanelDon1995,ostermann1992runge,Sanz-Serna1986,verwer1985convergence,wanner1996solving}, i.e., the numerical solution of certain stiff problems \cite{biswas2021structure} converges more slowly than what the formal order of the scheme would suggest. While there exist time-stepping methods that are devoid of order reduction, like linear multi-step methods (LMMs) \cite{lubich1990convergence}, the practical importance of RK methods (as well as related approaches that are equivalent to RK methods \cite{crouzeix1980methode}) renders the question ``how can order reduction be avoided in RK methods?'' central.

Order reduction may manifest in multiple shapes and forms. For explicit RK integration of mildly stiff IBVPs (e.g., advection), techniques to avoid the phenomenon have been developed in
\cite{abarbanel1996removal,alonso2002runge,alonso2004avoiding,CarpenterGottliebAbarbanelDon1995,pathria1997correct}. For stiff ODE problems, order reduction can be explained in terms of stiff limits \cite{prothero1974stability,wanner1996solving}. Order reduction in PDE IBVPs, first pointed out in \cite{crouzeix1975approximation,crouzeix1980approximation}, manifests in an interesting geometric fashion, in a way that the time-stepping error produces spatial boundary layers \cite{rosales2017spatial}.
Foundational work on the numerical analysis of order reduction includes \cite{lubich1993runge,lubich1995runge,Sanz-Serna1986,verwer1985convergence}, and rigorous error analysis for RK methods applied to linear PDEs has been developed in \cite{alonso2003optimal,gonzalez1999optimal,ostermann1992runge,scholz1989order}.

In the stiff setting, implicit Runge-Kutta (IRK) methods with high stage order \cite{wanner1996solving} can remedy the order reduction phenomenon. Unfortunately, high stage order requires a fully implicit RK structure, while the DIRK methods \cite{kennedy2016diagonally} are limited to low stage order \cite{ketcheson2018dirk}. Approaches aimed at bridging this gap include a weaker criterion than stage order that diminishes order reduction specifically for ROW methods applied to linear problems \cite{scholz1989order}. Similar conditions were proposed in \cite{ostermann1992runge}, albeit without providing numerical schemes that satisfy those conditions. In a similar spirit, the concept of weak stage order (WSO) was proposed in \cite{rosales2017spatial}, which is a generalization of the conditions stated in \cite{rang2014analysis}. More recently, \cite[Chapter 6]{Roberts2021} and \cite{RobertsSandu2022} extended the weak stage order conditions to generalized-structure additively partitioned Runge-Kutta (GARK) methods.

Like stage order, WSO imposes certain algebraic relations between the Runge-Kutta coefficients, but with two key differences: (a)~WSO remedies order reduction only for certain problems; however, (b)~the WSO conditions are compatible with the DIRK structure. A special case of WSO, called the WSO eigenvector criterion, has been studied in \cite{ketcheson2018dirk}: DIRK schemes up to order $4$ and WSO $3$ have been provided. At the same time, a barrier theorem was proved \cite{ketcheson2018dirk, BiswasKetchesonSeiboldShirokoff2022}: for high-order DIRK schemes, the WSO eigenvector criterion cannot be extended beyond WSO $3$.


\section{The Order Reduction Phenomenon}
\label{sec:OR_phenomenon}
In this section we review the order reduction phenomenon in the context of stiff ODEs. Prothero and Robinson \cite{prothero1974stability} introduced a family of problems of the form
\begin{equation}\label{pro_robin_stiff_problem}
    u^{\prime} = \lambda\left(u-\phi(t)\right)+\phi^{\prime}(t), 
    \quad u(0) = u_{0},\quad \text{with} \ \text{Re}(\lambda)\leq 0\;.
\end{equation}
Here $\phi(t)$ is any smooth function that varies at a moderate rate (i.e., $\phi'(t) = O(1)$), while $\lambda$ is a parameter that allows one to make the problem \eqref{pro_robin_stiff_problem} arbitrarily stiff. If $u(0)=\phi(0)$, then for any $\lambda \in \mathbb{C}$, \eqref{pro_robin_stiff_problem} has the solution $u(t) = \phi(t)$.  If $\textrm{Re}(\lambda) \ll -1$, the problem \eqref{pro_robin_stiff_problem} is stiff, and solutions different from $\phi(t)$ decay rapidly back to $\phi(t)$.

Equation \eqref{pro_robin_stiff_problem} provides a useful model for analyzing the truncation errors of a Runge-Kutta scheme for a stiff problem. One may introduce the following \emph{local truncation errors} (LTEs) \cite{wanner1996solving}: $E_{i,\Delta t}(t_n)$ for the intermediate stages, and $E_{\Delta t}(t_n)$ for the final step update. The LTEs characterize the failure of the exact solution $u(t)=\phi(t)$ to satisfy the RK scheme, and are obtained as the residuals of substituting $g_i = \phi(t_n+c_i \Delta t)$, $u_n = \phi(t_n)$, and $u_{n+1} = \phi(t_n+ \Delta t)$ into the RK scheme \eqref{RK_methods_1}--\eqref{RK_methods_2} applied to problem \eqref{pro_robin_stiff_problem}. Upon Taylor-expanding about $t_n$, the LTEs are \cite{biswas2021structure}:
\begin{align}\label{Eq:TaylorLTEStages}
    \vec{\mathcal{E}}(t_n) &= \sum_{k \geq 1} \frac{( \Delta t)^{k}}{(k-1)!} \vec{\tau}^{(k)} \phi^{(k)}(t_n) \;, \\
    \label{Eq:TaylorLTEFinal}
    E_{\Delta t}(t_n) &= \sum_{k \geq 1} \frac{( \Delta t)^{k}}{(k-1)!}\left[\sum_{j=1}^{s}b_{j}c_{j}^{k-1}-\frac{1}{k}  \right] \phi^{(k)}(t_n) \;,
\end{align}
where $\vec{\mathcal{E}}(t_n) := \left[E_{1,\Delta t}(t_n),E_{2,\Delta t}(t_n), \ldots, E_{s,\Delta t}(t_n)\right]^T$. Here the vector 
\begin{equation*}
    \vec{\tau}^{(k)} := A \vecpower{c}{k-1}-\frac{1}{k}\vecpower{c}{k}, \quad \textrm{for } k \geq 1\;,
\end{equation*}
is called the $k$th \emph{stage order residual} (it will play an important role later), $\phi^{(k)}(t_n) $ is the $k$th derivative of $\phi$ at $t_n$, and $\vecpower{c}{k} := \left[c_{1}^{k},c_{2}^{k}, \ldots, c_{s}^{k} \right]^T$ denotes component-wise exponentiation. Notice that \eqref{c-assumption} implies $\tau^{(1)}=0$.

The numerical approximation error at time $t_{n}$ is then defined as $\epsilon_{n}:= u_{n} - \phi(t_n)$. It satisfies the same linear recursion as the RK scheme, with a forcing prescribed by the LTEs \eqref{Eq:TaylorLTEStages} and \eqref{Eq:TaylorLTEFinal}:
\begin{equation}\label{final_rec_err_rel}
\epsilon_{n+1} = R(\zeta) \epsilon_{n} + \underbrace{\zeta \vecpower{b}{T}(I-\zeta A)^{-1} \vec{\mathcal{E}}(t_n)}_{=\Err} +E_{\Delta t}(t_n) \;.
\end{equation}
Here $\zeta := \lambda \Delta t$, and $R(z)$ ($z \in \mathbb{C}$) is the \emph{stability function}:
\begin{equation}\label{Eq:stabilityfunction1}
    R(z) := 1 + z\vecpower{b}{T} (I - z A)^{-1} \vec{e} = 
    \frac{\det(I-zA+z\vec{e}\vecpower{b}{T})}{\det(I-zA)}\;.
\end{equation}
Inspecting the expressions \eqref{Eq:TaylorLTEStages}, \eqref{Eq:TaylorLTEFinal} and \eqref{final_rec_err_rel} above, we see that the following conditions influence the order of the local error:
\begin{align}
    B(\xi):& & \vecpower{b}{T} \vecpower{c}{k-1} &= \frac{1}{k} & \text{ for } &k = 1,2,\ldots,\xi\;; \\
    C(\xi):& & \vec{\tau}^{(k)} &= 0  & \text{ for } &k = 1, 2, \ldots, \xi\;; \phantom{\dfrac{1}{k}}  \\
    \SCondAlb(\xi):& & \vecpower{b}{T} A^j \vec{\tau}^{(k)} &= 0 & \text{ for } &k > 0,\ j+k < \xi\;; \phantom{\dfrac{1}{k}} \label{Scond} \\
    T(\xi):& & \vecpower{b}{T} A^{k-1}\vec{e} &= \frac{1}{k!} & \text{ for } &k = 1, 2, \ldots, \xi\;.
\end{align}
The conditions $B(\xi)$, $C(\xi)$ are widely used and known as \emph{simplifying assumptions}; they determine the order of accuracy of  the quadrature and subquadrature rules on which the RK method is based \cite{butcher2008numerical}. Notice that $B(p)$ implies $E_{\Delta t} = \O(\Delta t^{p+1})$. Here we have introduced notation for the additional conditions $T$ and $S$ since they play an important role below. The conditions $T(\xi)$ determine the order of accuracy of the method for non-stiff linear problems. The conditions $\SCondAlb(\xi)$ have appeared for instance in \cite{albrecht1987}. The conditions $B(p)$, $S(p)$, and $T(p)$ are necessary (though not sufficient) for a method to be of order $p$ for general problems. Notice that if $B(p)$ and $T(p)$ hold, then the first and last terms in \eqref{final_rec_err_rel} are $\O(\Delta t^{p+1})$.  It remains only to bound the second term, 
\begin{equation*}
    \Err := \zeta \vecpower{b}{T}(I-\zeta A)^{-1} \vec{\mathcal{E}}(t_n)\;,
\end{equation*}
which is the one that causes order reduction in the stiff setting and on which we focus herein.

In the classical RK theory, i.e., in the non-stiff case, the scheme's convergence is studied in the limit $\Delta t \to 0$ with $\zeta = \mathcal{O}(\Delta t)$. A Neumann expansion in $|\zeta| \ll 1$ of $\zeta (I-\zeta A)^{-1} = \zeta I+\zeta^2 A+\zeta^3 A^2+ \cdots$, then leads to the terms like $\vecpower{b}{T}A^{\ell}\vec{\tau}^{(k)}$ in $\Err$ with $\ell \geq 0$, so that condition $S(p)$ guarantees the one-step error is $\O(\Delta t^{p+1})$.

In the case of stiff problems, we are interested in time steps that are large relative to the fastest time scale of the problem dynamics, which is  $\frac{1}{|\lambda|}$ for the Prothero-Robinson problem, i.e., we want $|\lambda|\Delta t \gg 1$. Hence, we study the convergence of errors under the simultaneous limits $\Delta t \to 0$ and $\zeta \to - \infty$, i.e., $\lambda \to - \infty$ faster than $\Delta t \to 0$. In this case $\zeta^{-1}$ is small and a Neumann expansion yields $\zeta (I-\zeta A)^{-1} = -A^{-1}(I-\zeta^{-1}A^{-1})^{-1}= -A^{-1} - \zeta^{-1} A^{-1}-\zeta^{-2} A^{-2}+ \cdots$, leading to the terms like $\vecpower{b}{T}A^{\ell}\vec{\tau}^{(k)}$ but with $\ell < 0$. These quantities are not guaranteed to vanish by the order conditions, and this in general leads to order reduction.

One way to avoid order reduction is to use schemes with high stage order.
\begin{definition}[Stage order]
     The stage order of a RK scheme is $q = \min\{q_1, q_2\}$,
    where $q_1, q_2$ are the largest integers such that $B(q_1)$ and $C(q_2)$ hold.
\end{definition}
Stage order $q$ implies that every stage of the scheme is an approximation accurate to at least order $q$,
and in particular that the method itself has order at least $q$.  Furthermore,
for a scheme with stage order $q$, it can be shown that the local error is $\O(\Delta t^{q+1})$ even in the stiff regime, thus avoiding order reduction. Unfortunately, DIRK schemes are restricted to low stage order; see e.g.~\cite{ketcheson2018dirk} for a proof of the following well-known result.
\begin{theorem}
	The stage order of an irreducible DIRK scheme is at most $2$. The stage order of a DIRK scheme with non-singular $A$ is at most $1$.
\end{theorem}

In the next section, we describe a criterion, called weak stage order (WSO), that is weaker than the stage order conditions but compatible with the DIRK structure. We show later that high order DIRK schemes with high WSO avoid order reduction for a certain class of problems, including the Prothero-Robinson problem \eqref{pro_robin_stiff_problem}.
\section{Weak Stage Order, Order Conditions, and Their Relationship}
\label{sec:notion_WSO}
Since DIRK schemes cannot have high stage order, a weaker condition, referred to as \emph{weak stage order} (WSO), was introduced \cite{ketcheson2018dirk,rosales2017spatial}. High weak stage order can alleviate order reduction in linear problems, and, in contrast to high stage order, is compatible with a DIRK structure. 

The idea behind WSO is to prescribe conditions on $(A,\vec{b})$ that increase the accuracy of the problematic error term $\Err$ in \eqref{final_rec_err_rel}, via the following fact: it holds that $\Err = O(\Delta t^{q+1})$, if
\begin{equation}\label{Eq:WSOMotivation}
    \vecpower{b}{T} (I - \zeta A)^{-1} \vec{\tau}^{(k)} = 0, \; \textrm{ for } k = 1, \ldots, q\;.
\end{equation}
Weak stage order is thus formulated to ensure that \eqref{Eq:WSOMotivation} holds. 

Consider the following space, which is the direct sum of Krylov subspaces generated by the stage order residuals $\{\vec{\tau}^{(1)},\vec{\tau}^{(2)},\ldots,\vec{\tau}^{(q)}\}$: 
\begin{equation}\label{Def:Kspace}
    K_q:=\operatorname{span}\left\{\vec{\tau}^{(1)}, A \vec{\tau}^{(1)}, \ldots, A^{s-1} \vec{\tau}^{(1)}, \vec{\tau}^{(2)}, A \vec{\tau}^{(2)}, \ldots, A^{s-1} \vec{\tau}^{({q})}\right\}.
\end{equation}
Note that by definition, and through application of the Cayley-Hamilton theorem, $K_q$ is an $A$-invariant subspace (that is $A \vec{v} \in K_q$ for any $\vec{v} \in K_q$). It will also be helpful to define the set of RK coefficients $(A, \vec{b})$ such that $\vec{b} \perp K_q$, namely:
\begin{equation}\label{def:AV_wso}
    \WSOset := \left\{ A \in \mathbb{R}^{s \times s}, \vec{b} \in \mathbb{R}^s \mid 
    \vecpower{b}{T} A^{j} \vec{\tau}^{(k)}=0, \; \text{for} \ 0 \leq j \leq s-1, \; 1 \leq k \leq q\right\}.
\end{equation}
\emph{Weak stage order} can be defined in one of two equivalent ways.
\begin{definition}(Weak stage order, version 1)\label{dfn:wso1}
	The weak stage order $q$ of an $s$-stage RK scheme $(A,\vec{b})$ is the largest integer for which $\vec{b} \perp K_q$, i.e., $(A, \vec{b}) \in \mathbb{W}_q$. If $\vec{b} \perp K_q$ holds for every $q \geq 1$, then $q = \infty$.
\end{definition}
The second abstract version makes use of invariant subspaces. 
\begin{definition}(Weak stage order, version 2)\label{dfn:wso2}
    The weak stage order $q$ of an $s$-stage RK scheme $(A,\vec{b})$ is the largest integer for which there exists an $A$-invariant vector space $V$ such that: $\vec{\tau}^{(k)} \in V$ for $1 \leq k \leq q$ ~and~ $\vecpower{b}{T} y = 0$ for all $y \in V$. If $\vec{\tau}^{(k)} \in V$ for all $k \geq 1$, then $q = \infty$.
\end{definition}
These two definitions of weak stage order are equivalent (i.e., taking $V = K_q$ in \cref{dfn:wso2}). Moreover, weak stage order is the most general criterion to guarantee \eqref{Eq:WSOMotivation}, thereby avoiding order reduction in (stiff) linear problems \cite{ketcheson2018dirk}.

In addition to the set of schemes that satisfy the WSO equations, we also introduce the set of $p$th order schemes:
\begin{equation}\label{def:AV_oc}
    \OCset := \left\{ A \in \mathbb{R}^{s\times s}, \vec{b} \in \mathbb{R}^s \; | \; (A,\vec{b}) \textrm{ satisfy all order conditions up to order } p \right\}. 
\end{equation}
A list of all order conditions up to $p = 5$ is given in Table~\ref{Table:ButcherOrderconditions}. 

We now discuss the polynomial equations defining RK schemes with order $p$ and WSO $q$, i.e., the set $\WSOset \cap \OCset$. Notice that the solutions to the WSO equations (i.e., schemes in $\mathbb{W}_q$) are of the same form as the conditions $\SCondAlb(\xi)$ \eqref{Scond}, but they are required to hold for a larger set of values $j, k$. Since the conditions \eqref{Scond} appear explicitly in the RK order conditions as formulated by Albrecht \cite{albrecht1987}, there is some overlap or redundancy between the conditions for WSO $q$ that define $\WSOset$ and the conditions for order $p$ that define $\OCset$.
In the more widely used formulation of RK order conditions due to Butcher (which we will also employ later), one instead has the related conditions (cf.~\eqref{Table:ButcherOrderconditions})
\begin{equation*} 
    \phi_{\ell,k} := \vecpower{b}{T}A^{\ell-k}\vecpower{c}{k} - \frac{k!}{(\ell+1)!} = 0
        \ \ \forall\, 0 \leq k \leq \ell \leq p-1\;.
\end{equation*}
The expressions appearing in the WSO conditions are just linear combinations of these
$\phi_{\ell,k}$:
\begin{equation*} 
    \vecpower{b}{T}A^j\vec{\tau}^{(k)} = \vecpower{b}{T}A^j\left( A \vecpower{c}{k-1}-\frac{1}{k}\vecpower{c}{k}\right) = \phi_{j+k,k-1}-\frac{1}{k}\phi_{j+k,k}
    \ \ \forall\, j \geq 0, k \geq 1\;.
\end{equation*}
Therefore, if a method has WSO $q$ then each of the conditions $\phi_{\ell,k}=0$ for $1\le k\le q$ is equivalent. When constructing RK schemes with high WSO, we can therefore pick just one $\phi_{\ell,k}=0$ from each equivalent family. Table~\ref{Table:redundancy_summary} summarizes which order conditions we keep (and which we discard as redundant) in the construction of $\WSOset \cap \OCset$.

Meanwhile, the WSO equations independently contain some redundancy; the Cayley-Hamilton theory used in \cref{dfn:wso1} overestimates the set of equations required to define an invariant subspace $K_q$. The following section discusses how to construct a low-dimensional subspace $K_q$ by removing redundant equations in the definition of $\WSOset$.

\begin{table}
\centering
\begin{tabular}{ |@{~}c@{~}|@{~}c@{~~}c@{~~}c@{~~}c@{~~}c| } 
\hline
    Order     & $T(p)$ & $B(p)$ & Related & \multicolumn{2}{c|}{Additional} \\ 
              &        &        & to $S(p)$ & \multicolumn{2}{c|}{Order Conditions}  \\
\hline
\hline
 $p=1$ & $\vecpower{b}{T} \vec{e}\phantom{^2} = 1$ & & & &\\ 
\hline
 $p=2$ & $\vecpower{b}{T}\!A \vec{e}\phantom{^2} = \frac{1}{2}$ & & & &\\ 
\hline
 $p=3$ & $\vecpower{b}{T}\!A^2 \vec{e} = \frac{1}{3!}$ & $\vecpower{b}{T} \vecpower{c}{2} = \frac{1}{3}$ & & &\\ 
\hline
 $p=4$ & $\vecpower{b}{T}\!A^3 \vec{e} = \frac{1}{4!}$ & $\vecpower{b}{T} \vecpower{c}{3} = \frac{1}{4}$ & 
 $\vecpower{b}{T}\!A \vecpower{c}{2} = \frac{1}{12}$  & $\vecpower{b}{T} C A \vec{c} = \frac{1}{8}$ &
 \\
\hline
 $p=5$ & $\vecpower{b}{T}\!A^4 \vec{e} = \frac{1}{5!}$ & $\vecpower{b}{T} \vecpower{c}{4} = \frac{1}{5}$ & 
 $\vecpower{b}{T}\!A^2\vecpower{c}{2}  = \frac{1}{60}$ & $\vecpower{b}{T}C^2\!A\vec{c}  = \frac{1}{10}$ & $\vecpower{b}{T}CA\vecpower{c}{2}  = \frac{1}{15}$ 
  \\ 
 & & & 
 $\vecpower{b}{T}\!A\vecpower{c}{3}  = \frac{1}{20}$ & $\vecpower{b}{T}CA^2\vec{c}  = \frac{1}{30}$ & $\vecpower{b}{T}\!ACA\vec{c}  = \frac{1}{40}$ \\
 & & &  & $\vecpower{b}{T}DA\vec{c}  = \frac{1}{20}$ &  \\
\hline
\end{tabular}
\vspace{.2em}
\caption{Order conditions (in Butcher's notation): Here $D = \mathrm{diag}(A\vec{c})$, $C = \mathrm{diag}(\vec{c})$, and $\vec{c} = A \vec{e}$.}
\label{Table:ButcherOrderconditions}
\end{table}


\begin{table}
\centering
    \begin{tabular}{ | c | c | c | c | c |}
      \hline
      \thead{$(p,q)$} & \thead{$n_p$} & 
      \thead{$\#$ redundant eq.}
      \thead{$\frac{(p-1)(p-2)}{2}$} & \thead{redundant $\phi_{j,k}$} & \thead{$\phi_{j,k}$ kept} \\
      \hline
            \hline
       $(3,2)$ & $4$  & $1$ & $\phi_{2,1}$ & $\phi_{1,1}$, $\phi_{2,2}$ \\
      \hline
       $(3,3)$ & $4$  & $1$ & $\phi_{2,1}$ & $\phi_{1,1}$, $\phi_{2,2}$ \\
      \hline
       $(4,3)$ & $8$  & $3$ & $\phi_{2,1},\phi_{3,2},\phi_{3,1}$ & 
       $\phi_{1,1}$, $\phi_{2,2}$, $\phi_{3,3}$\\
      \hline
       $(4,4)$ & $8$  & $3$ & $\phi_{2,1},\phi_{3,2},\phi_{3,1}$ & 
       $\phi_{1,1}$, $\phi_{2,2}$, $\phi_{3,3}$\\
      \hline
       $(5,4)$ & $17$  & $6$ & \makecell{$\phi_{2,1},\phi_{3,2},\phi_{4,3}$ \\ $\phi_{3,1},\phi_{4,2},\phi_{4,1}$} & 
       \makecell{$\phi_{1,1}$, $\phi_{2,2}$, $\phi_{3,3}$ \\ $\phi_{4,4}$} \\
      \hline
      $(5,5)$ & $17$  & $6$ & \makecell{$\phi_{2,1},\phi_{3,2},\phi_{4,3}$ \\ $\phi_{3,1},\phi_{4,2},\phi_{4,1}$} & 
      \makecell{$\phi_{1,1}$, $\phi_{2,2}$, $\phi_{3,3}$ \\ $\phi_{4,4}$}
       \\
      \hline
    \end{tabular}
\vspace{.2em}
\caption{Given order $p$ and WSO $q$, the numbers $n_p$ and $\frac{(p-1)(p-2)}{2}$ are the total, and redundant number of order condition, respectively. The last two columns show which $\phi_{\ell,k}$ we retain, vs.\ discard as redundant (since they are already implied by WSO and the retained conditions).}
\label{Table:redundancy_summary}
\end{table}

\section{Results from Weak Stage Order Theory}
\label{sec:WSOTheory_Results}
In this section we summarize key theoretical results from the companion paper \cite{BiswasKetchesonSeiboldShirokoff2022}---which we use here to construct DIRK schemes with WSO greater than 3. The main results consist of (i)~lower bounds on the number of stages required to obtain WSO greater than 3 (in terms of the order $p$); and (ii)~formulas for constructing $K_q$. 

To start, we first introduce the minimal polynomial for $K_q$ which plays a central role in the results. Let $d = \dim(K_q)$ denote the dimension of $K_q$ and $\vec{w}_j$, for $j=1,\ldots, d$, be a basis for $K_q$. Let
\begin{equation*}
W := \begin{pmatrix} \vec{w}_1 | \vec{w}_2 | \cdots | \vec{w}_d \\\end{pmatrix}\in \mathbb{R}^{s\times d}\;.
\end{equation*}
Since the column space of $W$ is $A$-invariant, there is a square matrix $B \in \mathbb{R}^{d\times d}$ such that
\begin{equation}\label{Eq:A_inv_space}
    A W = W B\;.
\end{equation}
Equation \eqref{Eq:A_inv_space} simply states that each vector $A w_j$ can be expressed as a linear combination of vectors $w_i$, for $i = 1, \ldots, d$. 

The \emph{minimal polynomial} \cite[Chapter~8 \& 9A]{Axler2015} (see also \cite[Chapter~XIV \S 2]{lang2002}) $p(x)$ of a matrix $B$ is the (non-zero) monic polynomial of smallest degree for which $p(B) = 0$.  While it is often the case (for instance when the eigenvalues of $B$ are distinct) that the minimal polynomial is the characteristic polynomial, in general $p(x)$ may be of lower degree than $\textrm{char}_B(x)$ when $B$ has repeated eigenvalues (e.g., if $B = I$ is the $s\times s$ identity matrix then $p(x) = x-1$ while $\textrm{char}_B(x) = (x-1)^s$).

We define the \emph{minimal polynomial} $P(x)$ \emph{of} $K_q$ as the minimal polynomial of $B$ in \eqref{Eq:A_inv_space}. Note that $P(x)$ is intrinsic to the subspace $K_q$ and remains invariant under a change of basis. That is, if $W' = W T$ for an invertible matrix $T$ is an alternative basis for $K_q$, then \eqref{Eq:A_inv_space} reads $A W' = W' B'$ where $B' = T^{-1} B T$ is just a conjugation of $B$. Since $P(B) = 0$ is equivalent to $P(B') = 0$, the $P(x)$ does not depend on the choice of basis for $K_q$. 

The minimal polynomial $P(x)$ of $K_q$ satisfies several important properties which follow from the linear algebra of matrices restricted to invariant subspaces. We summarize them here (without proof), along with their implications for DIRK schemes with WSO. 
\begin{enumerate}
    \item [a)] $P(x)$ is the lowest degree (non-zero\footnote{In the case when $K_q = \{ 0\}$, $P(x) = 1$ is the constant polynomial.}, monic\footnote{The highest-power coefficient is $1$.}) polynomial that satisfies 
    \begin{equation}\label{Eq:MinPoly}
        P(A) \vec{w} = 0, \quad \forall \vec{w} \in K_q\;.
    \end{equation}
    Due to the Krylov structure of $K_q$, relation \eqref{Eq:MinPoly} can be restated in terms of the vectors $\vec{\tau}_k$ as:
    \begin{equation}\label{Eq:MinPoly2}
        P(A) \vec{\tau}_k = 0, \quad \textrm{for} \; k = 2, \ldots, q\;.
    \end{equation}
    \item[b)] $P(x)$ divides the characteristic polynomial of $B$. Thus, 
    \begin{equation*} 
        \deg(P) \leq \dim(K_q)\;.
    \end{equation*}
    \item[c)] $P(x)$ divides the characteristic polynomial of $A$. Hence, every root of $P(x)$ is an eigenvalue of $A$. For DIRK schemes, the roots of $P(x)$ are then a subset of the diagonal entries of $A$, i.e., $\{a_{11}, a_{22}, \ldots, a_{ss}\}$.  
\end{enumerate}
We now may summarize the key results from \cite{BiswasKetchesonSeiboldShirokoff2022}. The first result is a limitation theorem on high WSO.
\begin{theorem}\label{Thm:WSOLimiation} 
    (from \cite{BiswasKetchesonSeiboldShirokoff2022})
    A DIRK scheme with invertible $A$ and minimal polynomial satisfying $\deg(P) \leq 1$ is limited to WSO $q \leq 3$. 
\end{theorem}
The practical implication of theorem \ref{Thm:WSOLimiation} is that WSO $q > 3$ requires a minimal polynomial $\deg(P) \geq 2$. The next theorem demonstrates how the WSO $q$ impacts the number of stages $s$ required to achieve a given order $p$. 
\begin{theorem}\label{Thm:BoundonDimK} 
    (from \cite{BiswasKetchesonSeiboldShirokoff2022})
    An $s$-stage DIRK scheme with $n_c$ distinct abscissa values, order $p \geq 1$, and weak stage order $q \leq 2n_c - 1$ (with $K_q$ and $P(x)$ defined in \eqref{Def:Kspace} and \eqref{Eq:MinPoly2}) satisfies
    \begin{equation*} 
        s - p + 1 - \sigma \geq \dim{(K_q)} \geq \left\lfloor \frac{q}{2}\right\rfloor,
    \end{equation*}
    where $\sigma = 1$ if the method is stiffly accurate, and $\sigma = 0$ otherwise. 
\end{theorem}
A classical result in RK theory (e.g.~\cite[Theorem~4.18]{wanner1996solving}) is that DIRK schemes have order limited to $p \leq s + 1$. Theorem \ref{Thm:BoundonDimK} highlights that the ``gap'' in this bound is exactly what enables WSO q. 

Our goal here is to construct schemes with $q > 3$. Motivated by the implications of theorem \ref{Thm:WSOLimiation} and theorem \ref{Thm:BoundonDimK}, we choose $\dim (K_q)$ as small as possible, that is $\dim(K_q) = 2$ (and hence $\deg(P) \leq 2$). Theorem \ref{Thm:BoundonDimK} then requires the number of stages to be $s \geq p + 1 + \sigma$, and limits $q$ to $q \leq 5$ (which suffices for this work; however, the theory also allows for $q>5$ if $\dim(K_q)>2$). Furthermore, the following theorem (also from \cite{BiswasKetchesonSeiboldShirokoff2022}) characterizes the roots of $P(x)$ when $q > 3$:

\begin{theorem}\label{Thm:SecondRootP}  (Minimal polynomial when $\deg(P) = 2$; \cite{BiswasKetchesonSeiboldShirokoff2022}) 
    Consider an irreducible DIRK scheme with invertible $A$ and WSO $q > 3$. If $K_q$ has a minimal polynomial with $\deg(P) = 2$, then 
    \begin{equation}\label{Eq:SecondRootP}
        P(x) = (x- a_{11})(x - a_{22})\;.
    \end{equation}
\end{theorem}
Note that $P(x)$ in \eqref{Eq:SecondRootP} is valid for both cases $a_{11} = a_{22}$ and $a_{11} \neq a_{22}$.

We turn our attention to constructing spaces $K_q =\operatorname{span}\{\vec{w}_{1},\vec{w}_{2}\}$ with $\dim(K_q) = 2$ and $\deg(P) = 2$ (otherwise, via theorem \ref{Thm:WSOLimiation}, $\deg(P) \in \{0, 1\}$ would result in WSO 3 or less). Theorem \ref{Thm:SecondRootP} requires:
\begin{equation}\label{Eq:Basis}
    (A-a_{11}I)(A - a_{22}I) \vec{w}_{j} = 0, \ \ j = 1,2\;.
\end{equation}
If $\vec{w}_{k}$ are chosen as eigenvectors of $A$, then (without loss of generality) the solution to \eqref{Eq:Basis} is exactly one of:
\begin{align}\label{Eq:FirstTwoEV_1}
        \textrm{When } &a_{11} \neq a_{22}: \quad \quad 
        A \vec{w}_{1}  = a_{11} \vec{w}_{1} \quad \textrm{and} \quad
        A \vec{w}_{2}  = a_{22} \vec{w}_{2}\;, \\
        \label{Eq:FirstTwoEV_2}
        \textrm{When } &a_{11} = a_{22}: \quad \quad 
        A \vec{w}_{1}  = a_{11} \vec{w}_{1} \quad \textrm{and} \quad
        A \vec{w}_{2} = a_{11} \vec{w}_{2} + \vec{w}_{1}\;.
\end{align}
Note that no other solution to \eqref{Eq:Basis} is allowed: neither of the monomials $(A - a_{11}I)$ or $(A - a_{22}I)$ in \eqref{Eq:Basis} can individually annihilate both vectors in $K_q$; otherwise the degree of $P$ would be $1$. This forces the vectors $\vec{w}_{j}$ ($j =1, 2$) to have distinct eigenvalues when $a_{11} \neq a_{22}$, or  $\vec{w}_{2}$ to be generalized eigenvectors when $a_{11} = a_{22}$. 

Finally, the space $K_q$ with $\dim(K_q) = 2$ has the form
\begin{equation}\label{Eq:SOR_eigenbasis}
    \vec{\tau}^{(k)} = \beta_1^{(k)}\vec{w}_{1}+\beta_2^{(k)}\vec{w}_{2}, 
    \; \textrm{for}\; k = 2,3,\ldots, q\;,
\end{equation}
where $\beta_1^{(k)}$, $\beta_2^{(k)}$ are unknown coefficients to be solved for (along with $A$). WSO then may be guaranteed if $\vec{b} \perp K_q$, that is:
\begin{equation}\label{Eq:orth_eigenbasis} 
    \vecpower{b}{T} \vec{w}_{1}=0, \ \ \vecpower{b}{T} \vec{w}_{2}=0\;.
\end{equation}
Together, we will use equations \eqref{Eq:FirstTwoEV_1} and \eqref{Eq:FirstTwoEV_2}, as well as equations \eqref{Eq:SOR_eigenbasis} and \eqref{Eq:orth_eigenbasis} as a (minimal) system of equations for weak stage order (with $\dim(K_q) = 2$).

As a final remark, we discuss a family of DIRK schemes that satisfy the WSO equations \eqref{Eq:FirstTwoEV_1} and \eqref{Eq:SOR_eigenbasis}, yet are reducible to smaller (equivalent) schemes. Identifying and avoiding reducible DIRK schemes is important for the  construction of schemes in the next section. 

We say a scheme is \emph{$r$-confluent} if the abscissas of its first $r$ stages all coincide, i.e., $c_1 = \cdots = c_r$. DIRK schemes that are $r$-confluent are equivalent (specifically, $S$-reducible \cite[Chapter IV.12]{wanner1996solving}) to a simpler DIRK scheme where one stage replaces stages 1 through $r$. Specifically, consider two DIRK schemes in block form
\begin{equation}\label{Eq:BlockMatrix}
    A = \begin{pmatrix}
        A_{11} & 0 \\
        A_{21} & A_{22} 
    \end{pmatrix} \; \textrm{with} \;
    \vec{b} = \begin{pmatrix}
        \vec{b}_1 \\
        \vec{b}_2
    \end{pmatrix},
    \; \textrm{and} \;
    A^* = \begin{pmatrix}
        a_{11} & 0 \\ 
        \vecpower{a}{*}_{21} & A_{22}
    \end{pmatrix} \; \textrm{with} \;
    \vecpower{b}{*} = \begin{pmatrix}
        b_1^* \\
        \vec{b}_2
    \end{pmatrix},
\end{equation}
where $\vecpower{a}{*}_{21} := A_{21} \vec{e}$, is a vector consisting of the row sums of $A_{21}$ and  $b_1^* = \vecpower{b}{T}_1 \vec{e}$ (where $\vec{e}$ here is of length $m=$ number of columns of $A_{21}$). Then, we have
\begin{lemma}\label{Lem:ETReducibility}
    Let $(A, \vec{b})$ be an $s$-stage DIRK scheme with block structure given in \eqref{Eq:BlockMatrix}. If $(A_{11},\vec{b}_1)$ is an $r$-confluent scheme with $r$ stages where $2 \leq r \leq s$, then $(A, \vec{b})$ is reducible to $(A^*, \vecpower{b}{*})$.
\end{lemma}
\begin{proof}
      Applying Definition~12.17 in \cite{wanner1996solving}, where the partition of equivalent stages (i.e., partition of the integers $\{1, \ldots, s\}$) is taken as $S_{1} = \{1, 2, \ldots, r\}$ and $S_2 = \{r+1\}$, $\ldots$, $S_{s - r} = \{s\}$ shows that the scheme is $S$-reducible. Then \cite[Theorem~2.2]{hundsdorfer1980note} implies that the first $r$ stages of $A$ yield the same intermediate stage value solutions---and thus can be replaced by a single stage.
\end{proof}

\section{Optimization Problem for Finding DIRK Schemes with Desirable Properties}
\label{sec:optProblemDIRK}
In this section we formulate and numerically solve the problem of constructing DIRK schemes with a prescribed order $p$ and WSO $q$ that are A-stable, stiffly accurate (and hence L-stable), and have an optimally small error constant. A-priori, the degrees of freedom in this optimization problem are the coefficients in the matrix $A \in \mathbb{R}^{s \times s}$ and the vector $\vec{b} \in \mathbb{R}^s$. The constraints are as follows:\vspace{.2em}
\begin{enumerate}
    \item[(C.0)] (DIRK structure) $a_{ij} = 0$ for $j>i$;
    \item[(C.1)] ($p$th order conditions) $(A, \vec{b}) \in \OCset$, defined in \eqref{def:AV_oc};
    \item[(C.2)] (weak stage order $q$) $(A, \vec{b}) \in \WSOset$, defined in \eqref{def:AV_wso};
    \item[(C.2')] (WSO if $\dim(K_q) = 2$) $(A, \vec{b})$ satisfy \eqref{Eq:SOR_eigenbasis}, \eqref{Eq:orth_eigenbasis}, and either \eqref{Eq:FirstTwoEV_1} or \eqref{Eq:FirstTwoEV_2};
    \item[(C.3)] (stiff accuracy) $a_{sj} = b_{j}$ for $j = 1,\dots, s$;
    \item[(C.4)] (non-negative abscissas) $c_j \geq 0$ for $j = 1, \dots, s$;
    \item[(C.5)] (A-stability condition 1) $a_{ii} \geq 0$ for $1 \leq i \leq s$;
    \item[(C.6)] (A-stability condition 2) $|R(i y)| \leq 1$ for all $y \in \mathbb{R}$;
    \item[(CR.6)] (relaxation of (C.6)) $|R(i y)| \leq 1$ for $y \in \{y_1, \dots, y_{m} \}$ with $0 \leq y_1 < \ldots < y_{m}$. 
\end{enumerate}

\smallskip
Condition (C.0) enforces the DIRK structure. Conditions (C.1) and (C.2) are simply the order conditions and weak stage order conditions, respectively, while (C.2') is a simplified set of WSO conditions (based on \S\ref{sec:WSOTheory_Results}) for the special case $\dim(K_q) = 2$. Condition (C.3) ensures stiff accuracy and guarantees that the numerical solution is exact in the limit $\Delta t \to 0$ and $\zeta \to -\infty$ \cite{wanner1996solving}. It has the effect of prescribing $\vec{b}$ in terms of $A$ so that the degrees of freedom are the matrix $A$ only. Constraint (C.4) ensures that evaluations of $f(t,u)$ in \eqref{eq:IVP} do not occur prior to the initial time.

Lastly, conditions (C.5) and (C.6) impose A-stability (see \cite[Chapter IV.3 Eqs.~(3.6), (3.7)]{wanner1996solving}), which, combined with (C.3), ensures L-stability. Condition (C.6) is, as written, an infinite set of constraints. While it is possible to recast (C.6) as a finite set of inequalities involving semi-definite matrices (using a connection between non-negative single-variable polynomials and polynomials written as a sum of squares), here we take the simpler approach of approximating (C.6) by imposing it only at a finite set of values on the imaginary axis. Taking advantage of the symmetry $|R(i y)| = |R(-i y)|$ leads to the weaker set of constraints (CR.6).

For the purpose of constructing DIRK schemes with high WSO, we denote the set of equality constraints $\mathcal{C}_{\mathrm{Eq}}$ and inequality constraints $\mathcal{C}_{\mathrm{InEq}}$ by:
\begin{align*}
        \mathcal{C}_{\mathrm{Eq}} &:= \{ (A, \vec{b}) \in \mathbb{R}^{s\times s} \times \mathbb{R}^{s} \;| \; (C.0), (C.1), (C.2'), \textrm{and } (C.3) \textrm{ hold}\}\;, \\
        \mathcal{C}_{\mathrm{InEq}} &:= \{ (A, \vec{b}) \in \mathbb{R}^{s\times s} \times \mathbb{R}^{s} \;| \; (C.4), (C.5), \textrm{and } (CR.6) \textrm{ hold}\}\;.
\end{align*}
A \emph{feasible scheme} is one that satisfies both sets of constraints.

To guide the construction of DIRK schemes that achieve a minimal error, we use the $\ell^2$-norm of the residuals of the $(p+1)$st order conditions as a proxy for the error, leading to the objective function:\vspace{.2em}
\begin{enumerate}
     \item [(Ob)] (Objective function) $F(A,\vec{b}) := ||{(p+1)\textrm{st order conditions}}||_{\ell^2}^2$.
\end{enumerate}
\smallskip
For instance, for $p = 1$ and $p = 2$, the objective function (using $\vec{c} = A\vec{e}$) is:
\begin{align*}
&p = 1: \quad F(A, \vec{b}) = \left( \vecpower{b}{T} A \vec{e} - \tfrac{1}{2} \right)^2, \\
&p = 2: \quad F(A, \vec{b}) = 
\left( \vecpower{b}{T} A^2 \vec{e} - \tfrac{1}{6} \right)^2 + 
\left( \vecpower{b}{T} \vecpower{c}{2} - \tfrac{1}{3} \right)^2.
\end{align*}
If enough degrees of freedom are allowed, then a feasible scheme of order $p$ may satisfy $F(A, \vec{b}) = 0$, and thus be of order $p+1$. However, generally, locally optimal schemes will not satisfy $F(A, \vec{b}) = 0$ exactly. Altogether, we seek optimal DIRK schemes via the following constrained minimization problem: 
\begin{align*}
    \hspace{-2em} \textcolor{black}{(M)} \hspace{2em} \textrm{\rm Minimize~~~} & 
    F(A, \vec{b}) \\
    \textrm{\rm Subject to~}
    & \textrm{(C.0), (C.1), (C.2'), (C.3) and} \\
    & \textrm{(C.4), (C.5), (CR.6)}\;.
\end{align*}

\subsection{Solution to Problem (M)} 
\label{sec:DIRKsWithHighWSOConstruction}
While one can attempt to solve problem $\textcolor{black}{(M)}$ directly via black-box optimization routines, numerical experiments revealed that such a direct approach becomes highly inefficient as $s$, $p$, and $q$ are increased. Plausibly caused by the problem's lack of convexity and ill-conditioned constraints, feasible, let alone optimal, solutions are found increasingly rarely with increasing $s$, $p$, and $q$. In order to facilitate a more robust approach, we instead propose to solve $\textcolor{black}{(M)}$ in two major steps: We first construct a feasible scheme; then, using the feasible scheme as initial guess, we apply a local optimizer to minimize the objective function. 

\smallskip
\emph{1.~Construction of a feasible scheme}: We solve a sequence of sub-problems to find a feasible scheme satisfying both $\mathcal{C}_{\mathrm{Eq}}$ and $\mathcal{C}_{\mathrm{InEq}}$. Steps (1A) and (1B) construct a solution satisfying the equality constraints only. Step (1C) then incorporates the inequality constraints as well. 
  
  \noindent\fbox{Step (1A).} This step utilizes a hybrid analytical and numerical approach to find a point $(A, \vec{b})$ in $\mathcal{C}_{\mathrm{Eq}}$. Below, the substeps~(a)--(e) solve the first $(s-1)$ rows of equation \eqref{Eq:SOR_eigenbasis} and make use of the theory from \S\ref{sec:WSOTheory_Results}; then substep~(f) solves the remaining constraints. All numerical solutions in this step use $\mathrm{MATLAB}$'s \verb|sqp| algorithm in \verb|fmincon|, prescribing a constant objective function in order to use it simply as an algebraic solver.
  \begin{itemize}
  		\item[(a)] Solve (analytically) the first two components of the eigenvectors $\vec{w}^{(1)}$, $\vec{w}^{(2)}$: there are two solution branches, one corresponding to \eqref{Eq:FirstTwoEV_1} and another \eqref{Eq:FirstTwoEV_2}.  We restrict our solutions to the branch \eqref{Eq:FirstTwoEV_1} where $a_{11} \neq a_{22}$. The alternative case, $a_{11} = a_{22}$ is also possible but not pursued here.
  		\item[(b)] Solve (analytically) for $\beta_1^{(k)}$, $\beta_2^{(k)}$ in equation \eqref{Eq:SOR_eigenbasis}: the first two rows of \eqref{Eq:SOR_eigenbasis} uniquely define $\beta_1^{(k)}$, $\beta_2^{(k)}$, and are then automatically satisfied. 
  		\item[(c)] Solve (analytically) for the third component/row of \eqref{Eq:SOR_eigenbasis}: the $3$ equations in six variables $(a_{11}, a_{21}, a_{22}, a_{31}, a_{32}, a_{33})$ for $q = 4$ can be solved by parameterizing $a_{21}$, $a_{31}$, and $a_{32}$ in terms of $a_{11}$, $a_{22}$, and $a_{33}$. We choose the parameterized branch to avoid the reducible $r$-confluent schemes (see \S\ref{sec:WSOTheory_Results}).
  		\item[(d)] Construct (numerically) the upper $3\times3$ block of $A$ with numerical entries. For $q = 4$, we select $a_{11}$, $a_{22}$, and $a_{33}$ randomly and use the parameterization in (c) to determine $a_{21}$, $a_{31}$, $a_{32}$. For $q = 5$, we find a numerical solution to the third component/row of \eqref{Eq:SOR_eigenbasis} for $\vec{\tau}^{(5)}$ (which via the parameterization is an equation in terms of $a_{11}$, $a_{22}$, $a_{33}$).
  		\item[(e)] Solve (numerically) row by row (from row $4$ through $s-1$), the $(q-1)$ equations in \eqref{Eq:SOR_eigenbasis}. When $q =4,5$, the $r$th row yields $(q-1)$ ($\leq r$) equations in $r$ variables. At each row, we numerically find a solution.  At the end of this substep, the upper $(s-1)\times (s-1)$ block of $A$ is populated with numerical entries. 
  		\item [(f)] To satisfy (C.3) we set $\vecpower{b}{T} = (a_{s1},a_{s2},\dots,a_{ss})$, then solve (numerically) for the last row of $A$. Together, this amounts to solving the $(q-1)$ equations from row $s$ in \eqref{Eq:SOR_eigenbasis} (i.e., to satisfy (C.2')) and the non-redundant (cf.~\S\ref{sec:notion_WSO}) order conditions (C.1),
  		 \begin{equation}\label{subset_order_cond_4}
            \vecpower{b}{T}[\vec{w}^{(1)},\vec{w}^{(2)},\vec{e},\vec{c},\vecpower{c}{2},\vecpower{c}{3},CA\vec{c}] = \left[0,0,1,\tfrac{1}{2},\tfrac{1}{3},\tfrac{1}{4},\tfrac{1}{8}\right], 
        \end{equation}
        for order $p = 4$. Order $p = 5$ requires, in addition to \eqref{subset_order_cond_4}:
        \begin{equation*}
            \vecpower{b}{T}[\vecpower{c}{4},C^2A\vec{c},CAC\vec{c},\text{diag}(A\vec{c})A\vec{c},CA^2\vec{c},ACA\vec{c}] =  \left[\tfrac{1}{5},\tfrac{1}{10},\tfrac{1}{15},\tfrac{1}{20},\tfrac{1}{30},\tfrac{1}{40}\right].
        \end{equation*}
        For instance, a $4$th order DIRK scheme with WSO $4$, yields $10$ equations ($7$ from \eqref{subset_order_cond_4} and $q-1 = 3$ from row $s$ of \eqref{Eq:SOR_eigenbasis}) in $s$ variables.
  \end{itemize}
  \smallskip
  
   This procedure generates a random scheme that satisfies the equality constraints. For robustness purposes, we choose to reject (and simply re-start the step) any scheme that has a coefficient larger than 20 in absolute value (in line with \cite{sharp1993explicit}) or fails to satisfy the constraints to within $10^{-10}$.

  \medskip
  \noindent\fbox{Step (1B).} Step~(1A) uses \verb|fmincon| as a solver with (for computational speed) the residual error tolerance set significantly larger than machine precision. This generates a first approximation  $(A,\vec{b})$ to constraints (C.1), (C.2'), and (C.3). To drive the residuals down to machine precision, we use the output from step~(1A) as a starting point and solve the equations defined by $\mathcal{C}_{\mathrm{Eq}}$ (again) via a Gauss-Newton iteration.
  
  \medskip
  \noindent\fbox{Step (1C).} In this step we reincorporate the inequality constraints: non-negative abscissae (C.4) and A-stability ((C.5) and (CR.6)) to construct (fully) feasible schemes.  Using the output $(A,\vec{b})$ from step~(1B), we call \verb|fmincon| with the full constraint set and (again) a constant objective function. The resulting schemes turn out to satisfy the equality constraints to machine precision, and we observe that they tend to not lie on the boundary of the inequality constraints $\mathcal{C}_{\mathrm{InEq}}$. In the occurrence that the optimization solve in this step fails, we simply restart from step (1A).

\begin{remark} (SDIRKs and WSO) Restricting to the solution branch defined by \eqref{Eq:FirstTwoEV_1} in (1A) rules out singly diagonally implicit Runge-Kutta schemes (SDIRKs), which have all their diagonal entries identical. SDIRK schemes are of practical interest for their ease of implementation. While we defer the study of SDIRK schemes to later work, a preliminary exploration of the other solution branch \eqref{Eq:FirstTwoEV_2} (albeit with $\dim(K_q) > 2$) revealed that SDIRKs with high WSO do exist, demonstrating that the SDIRK structure is compatible with high WSO.
\end{remark}

\emph{2.~Optimization:} 
We will refer to schemes by the triple $(s,p,q)$, representing the number of stages, classical order, and weak stage order,
respectively.
According to theorem \ref{Thm:BoundonDimK}, a $4$th order DIRK scheme with WSO $4$ requires at least five stages, and a $5$th DIRK scheme with WSO $4$ or $5$ requires at least six stages.  These are lower bounds that may not be sharp, and the additional requirements we have imposed (such as A-stability and stiff accuracy) are likely to further increase the minimum viable number of stages.
In numerical searches, we have found schemes of type $(7,4,4)$, $(12,5,4)$, $(12,5,5)$.  Numerical searches failed
to find methods with the corresponding $p,q$ and fewer stages $s$.

We repeatedly (100,000+ times) solve $\textcolor{black}{(M)}$ via local optimization, starting with initial guesses given by the output of step~$(1\mathrm{C})$. This yields a set of locally optimal schemes.  We use MATLAB's \verb|fmincon| with the gradient-based \verb|sqp| algorithm.
Among the locally optimal schemes found in this manner, we have selected one from each class that is close to optimal in terms of $F(A, \vec{b})$ and is Pareto-optimal in terms of minimizing $F(A, \vec{b})$ and minimizing $\max_{i,j} |a_{ij}|$.
We thus provide three schemes, one for each triple $(s,p,q)$: DIRK-$(7,4,4)$, DIRK-$(12,5,4)$, and DIRK-$(12,5,5)$. Since we used the relaxation (CR.6) in place of (C.6), we check \emph{a posteriori} that the schemes are in fact A-stable. The stability regions and magnitude of the stability function along the imaginary axis, shown in Figure~\ref{fig:stability_region_1}, confirm this.
Scheme coefficients are given in \cref{app:butcherTableau}.

\begin{figure}[htb]
	\begin{minipage}[b]{.49\textwidth}
	    \centering
		\includegraphics[width=.86\textwidth]{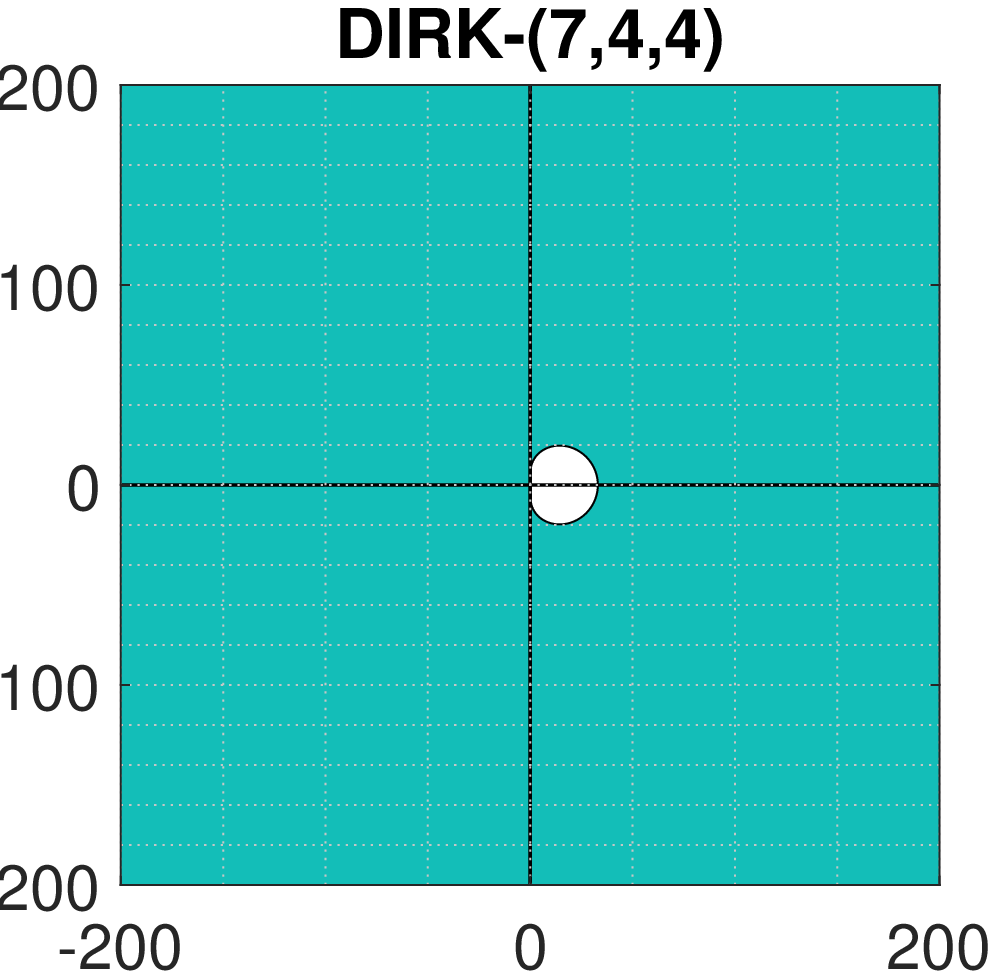}
	\end{minipage}
	\hfill
	\begin{minipage}[b]{.49\textwidth}
	    \centering
		\includegraphics[width=.86\textwidth]{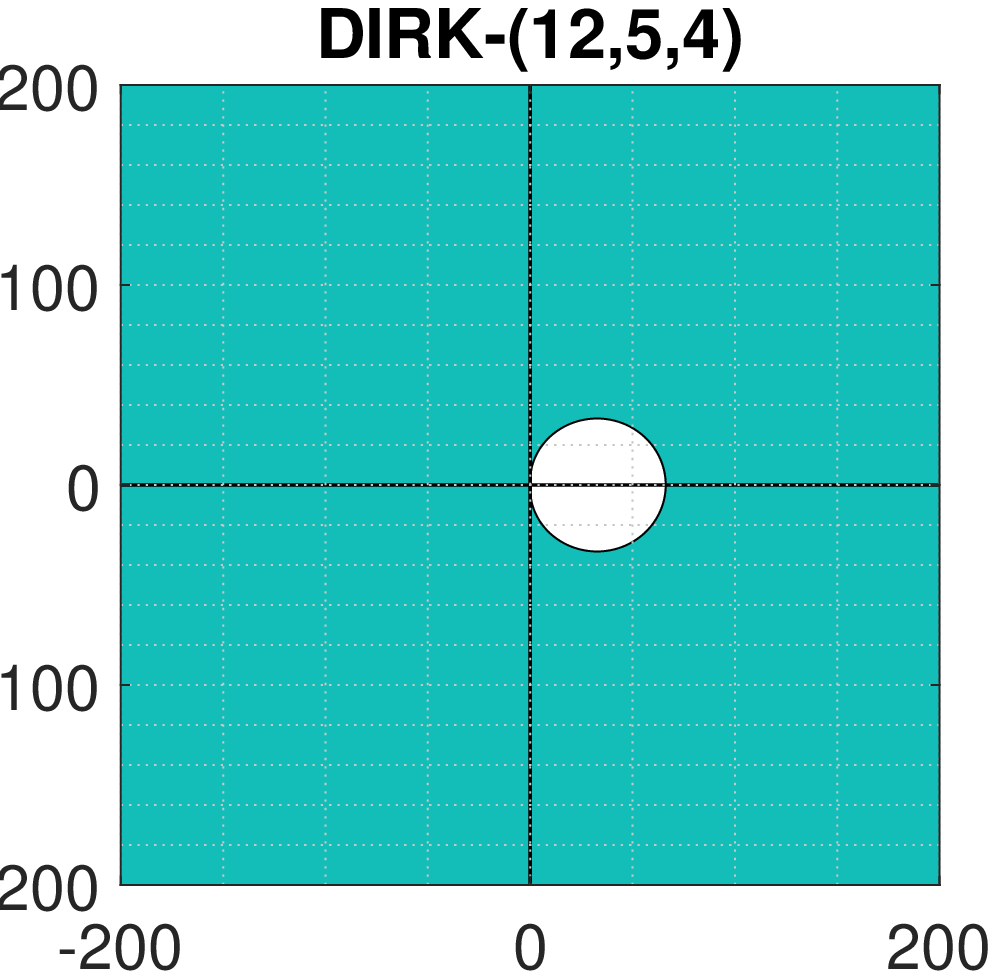}
	\end{minipage}
	
	\begin{minipage}[b]{.49\textwidth}
	    \centering
	    \includegraphics[width=.86\textwidth]{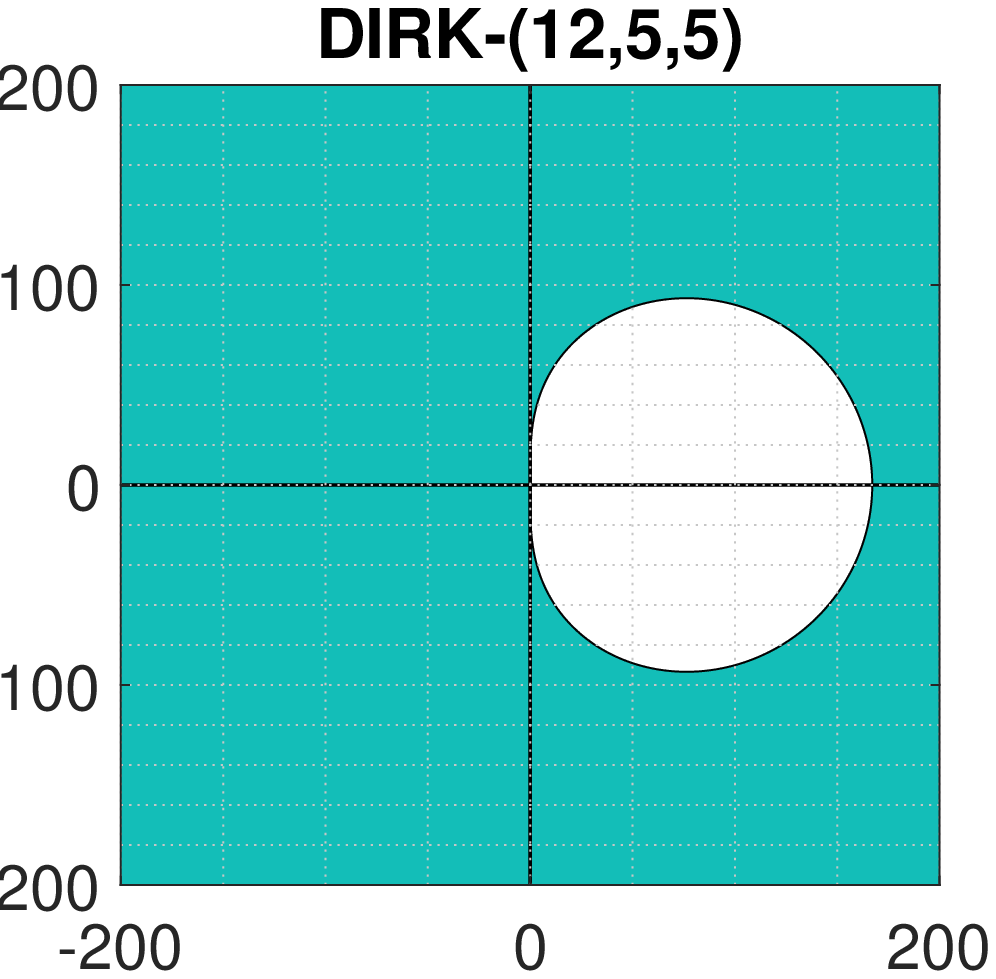}
    \end{minipage}
    \hfill
    \begin{minipage}[b]{.49\textwidth}
	    \centering
    	\includegraphics[width=\textwidth]{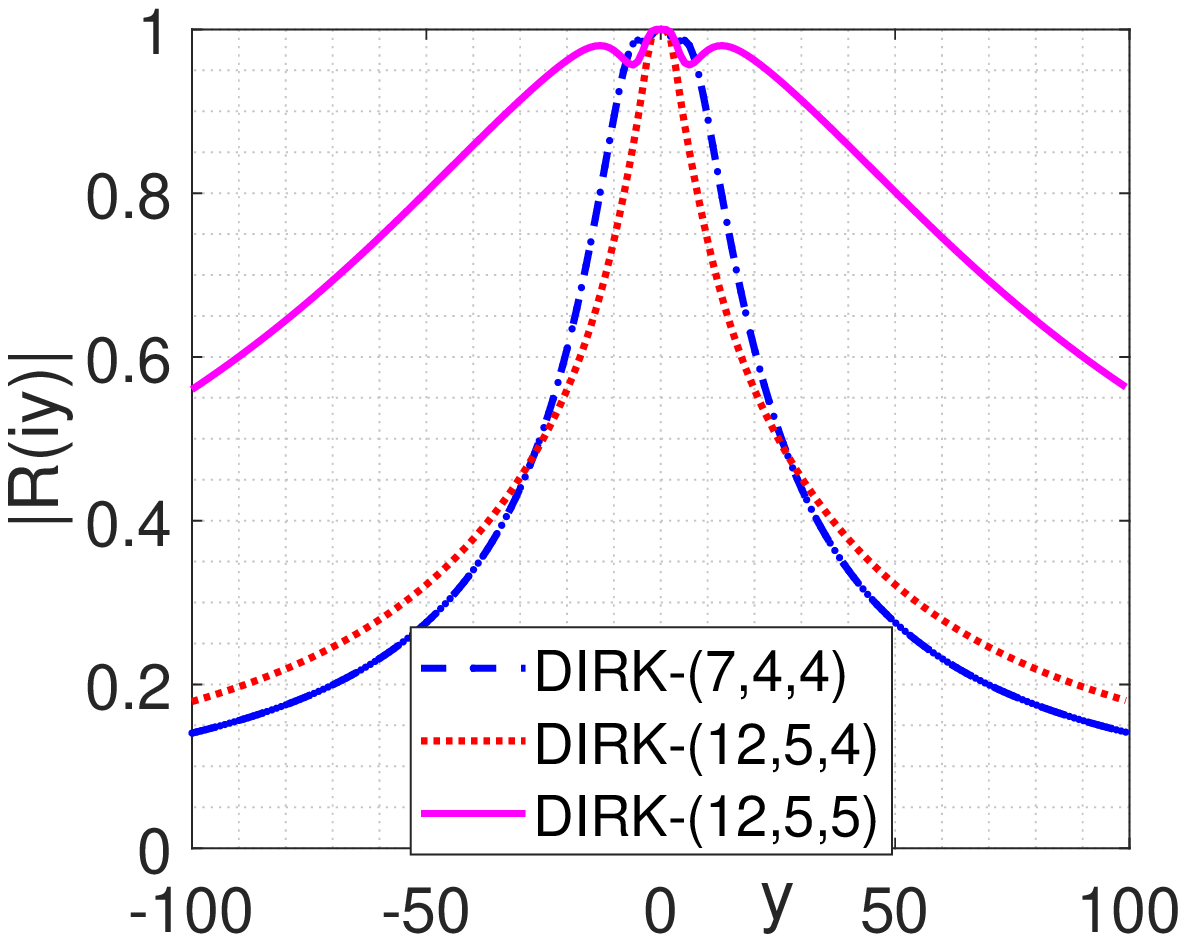}
    \end{minipage}
	\vspace{-.5em}
	\caption{Stability regions (in $\mathbb{C}$) for the new schemes, DIRK-$(7,4,4)$, DIRK-$(12,5,4)$, and DIRK-$(12,5,5)$, and $|R(iy)|$ for $y \in \mathbb{R}$ for the three schemes. Observe that $|R(iy)| \leq 1$ along the imaginary axis, hence these methods are A-stable.}
	\label{fig:stability_region_1}
\end{figure}

\section{Numerical Results: Linear Problems with Autonomous Operators}
\label{sec:numerical_result_linear_problems}
This section presents numerical test cases for ODE and PDE problems with linear operators with time-independent coefficients (the forcing and solutions may be time-dependent). This is the class of problems for which WSO is expected to alleviate order reduction. 

There are only a handful of theoretical results characterizing for which problems or PDEs the convergence rate can be guaranteed to be equal to the weak stage order.  For example, Ostermann and Roche \cite{ostermann1992runge} examined linear boundary value problems, i.e., $u_t = \mathcal{L} u + f$ with boundary condition $\mathcal{B}u = 0$, where $\mathcal{L}$ has a complete $L^2$ eigenfunction basis with (point spectrum) eigenvalues satisfying $\textrm{Re}(\lambda) \leq 0$ (see Assumptions~(3.1) in \cite{ostermann1992runge}). Here $\mathcal{L}$ may have coefficients that depend on space $x$, but not on time $t$; and $f(t)$ may be time-dependent. Then RK schemes satisfying the condition
\begin{equation}\label{eq:cond_Wk}
	W_k(z) \equiv 0 \quad \textrm{for } 1 \leq k \leq q\;, \quad \quad \textrm{where} \quad W_k(z) := \frac{k \vecpower{b}{T} (I - z A)^{-1} \vec{\tau}^{(k)} }{R(z) - 1}\;,
\end{equation}
along with Assumptions~(2.9) in \cite{ostermann1992runge}, overcome order reduction. Condition \eqref{eq:cond_Wk} is (essentially) implied by WSO $q$. In a similar spirit, condition \eqref{eq:cond_Wk} remedies order reduction for Rosenbrock methods \cite{OstermannRoche1993} in a more abstract setting where $\mathcal{L}$ is the infinitesimal generator of an analytic semi-group. While several of the PDEs we test here fall under the framework of known convergence results in \cite{ostermann1992runge}, some do not, such as the linear advection equation in \S\ref{subsec:linadveq}.

Below, the new schemes are denoted by $(s,p,q)$, where $s$ is the number of stages, $p$ is the classical order, and $q$ is the scheme's weak stage order. Methods with $q = p$ for ODEs and $q = p-1$ for PDEs yield solutions that converge at the rate $p$.  However, for PDEs, spatial derivatives of the solution may still exhibit order reduction if $q = p-1$. Methods with $q = p$ for PDE problems also alleviate order reduction in the solution's derivatives \cite{ostermann1992runge, rosales2017spatial}.

As references of comparison for our newly devised (high WSO) schemes, we include two schemes with WSO $q=1$, referred to as DIRK-($5,4,1$) \cite[Chapter~IV.6, Table~6.5]{wanner1996solving} which is A-stable and stiffly accurate; and DIRK-($5,5,1$) \cite[Table~24, p.~98]{kennedy2016diagonally}, which is A-stable but not stiffly accurate.

In each PDE test problem below, a spatial approximation is chosen so that the spatial approximation error becomes negligible relative to the temporal error. Hence, the error convergence plots below isolate the temporal error generated by the different DIRK schemes with high (and low) WSO. Note that for the different test problems, different spatial approximation strategies and numbers of grid points are employed to achieve this objective while also balancing simplicity and computational efficiency.

\subsection{Prothero-Robinson ODE test problem}
We study the problem \eqref{pro_robin_stiff_problem} with true solution $\phi(t)=e^{-t}\sin(10t)+\cos(20t)$, stiffness parameter $\lambda=-10^4$, initial condition $u(0)=\phi(0)$ and final time $T=10$. Figure~\ref{linear_ode_test_problem} contrasts high versus low WSO schemes, i.e., DIRK-(7,4,4) vs.~DIRK-(5,4,1), as well as DIRK-(12,5,5) vs.~DIRK-(5,5,1). For each scheme we observe convergence order $p$ for small enough $\Delta t$. However, in line with the theoretical predictions, for the schemes with $q<p$ the convergence rate is lower (approximately equal to $q$) for larger values of $\Delta t$, i.e., in the stiff regime.

\begin{figure}[htb]
	\begin{minipage}[b]{.47\textwidth}
	    \centering
		\includegraphics[width=.85\textwidth]{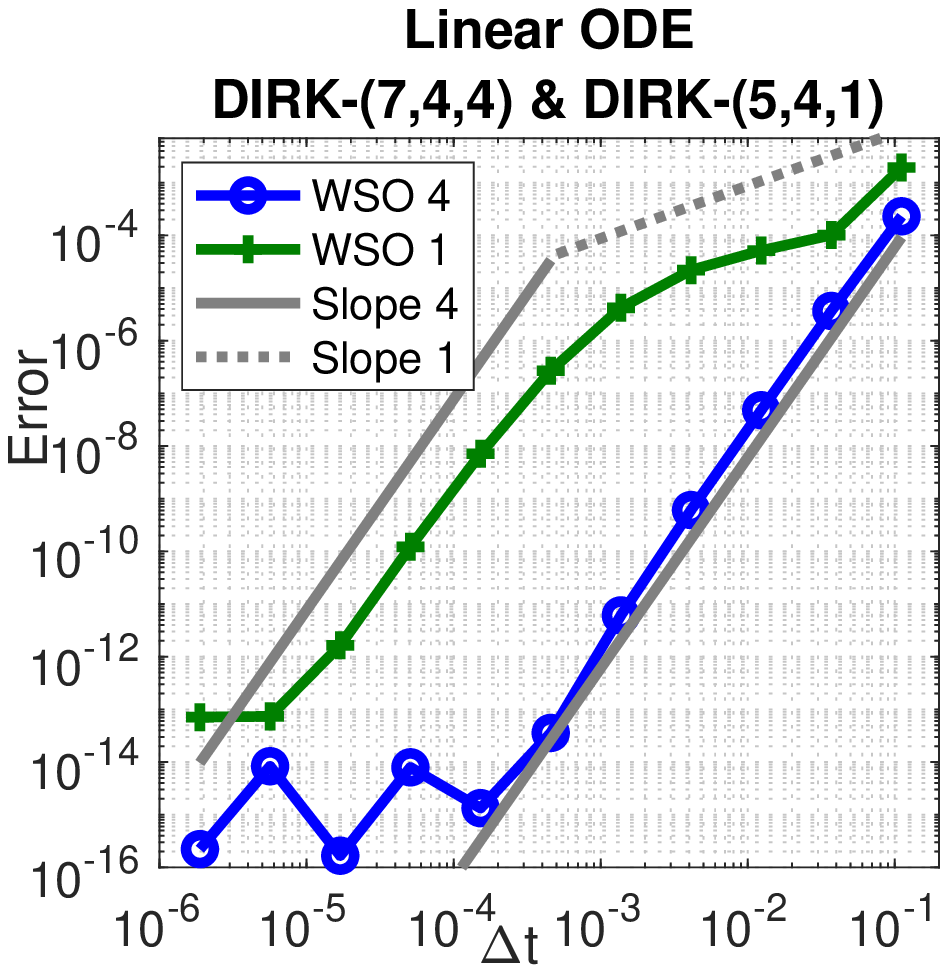}
	\end{minipage}
	\hfill
	\begin{minipage}[b]{.47\textwidth}
	    \centering
		\includegraphics[width=.85\textwidth]{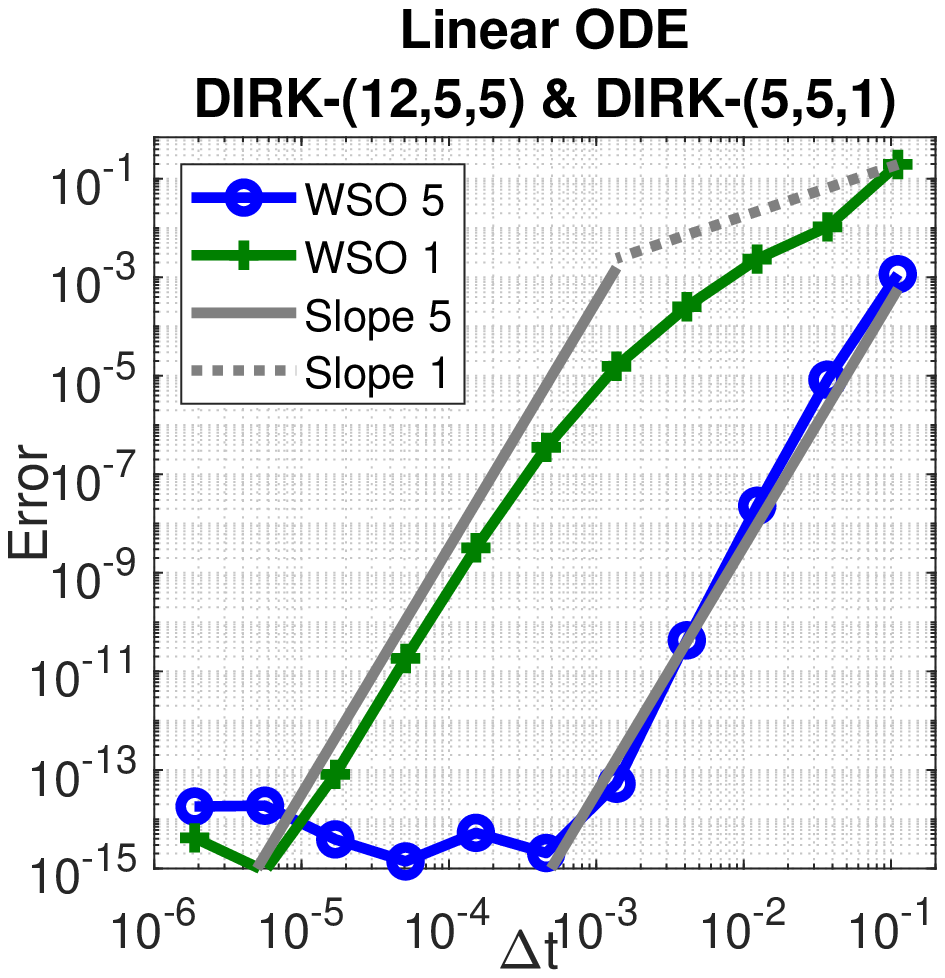}
	\end{minipage}
	\vspace{-.5em}
	\caption{Convergence for the Prothero-Robinson test problem using DIRK-$(7,4,4)$: $4$th order DIRK scheme with WSO $4$ (blue circles) and WSO $1$ (green)  (left), and DIRK-$(12,5,5)$: $5$th order DIRK scheme with WSO $5$ (blue circles) and WSO $1$ (green) (right).}
	\label{linear_ode_test_problem}
\end{figure}

\subsection{Heat equation}
Next we consider the $1$D heat equation
\begin{equation*}
u_t = u_{xx}+f \ \ \text{for} \ (x,t)\in (0,1) \times (0,1], \quad u=g(x,t) \ \ \text{on} \ \{0,1\} \times (0,1]\;,
\end{equation*}
with the forcing $f(x,t)$, the boundary conditions (b.c.) and the initial condition (i.c.) chosen such that $u(x,t) = \cos(20t)\sin(10x+10)$. To isolate the temporal error, we use a $4\text{th}$ order centered finite difference approximation in space on a grid with $10^4$ points. Errors are computed at the final time $T=1$ using the maximum norm in space. Figure~\ref{HeatEqn} shows the convergence of function values $u$ and derivatives $u_x$ using the three new high WSO DIRK schemes, compared with reference WSO-1 DIRK schemes of the respective orders. In agreement with the analysis in \cite{rosales2017spatial}, for this second-order PDE, the time stepping schemes produce spatial boundary layers (BLs) of width $\mathcal{O}(\Delta t^{0.5})$, resulting in a loss of half an order in $u_x$ when $q<p$. The results confirm the full order of convergence in $u$ and $u_x$ when using DIRK-$(7,4,4)$ and DIRK-$(12,5,5)$, and the full order in $u$ and half order loss in $u_x$ with DIRK-$(12,5,4)$. Note that with the given setup, the spatial approximation error is about $10^{-11}$, hence the stagnation of the errors around that value.

\begin{figure}[htb]
	\begin{minipage}[b]{.32\textwidth}
		\includegraphics[width=\textwidth]{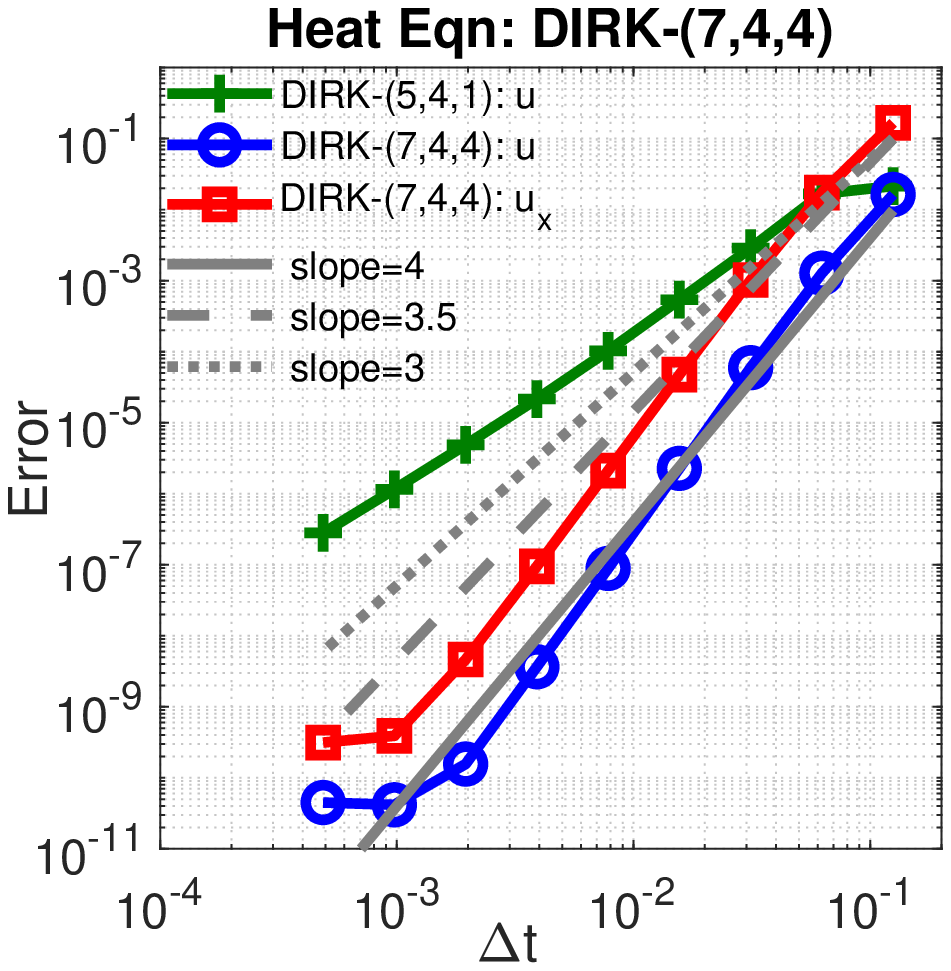}
	\end{minipage}
	\begin{minipage}[b]{.32\textwidth}
		\includegraphics[width=\textwidth]{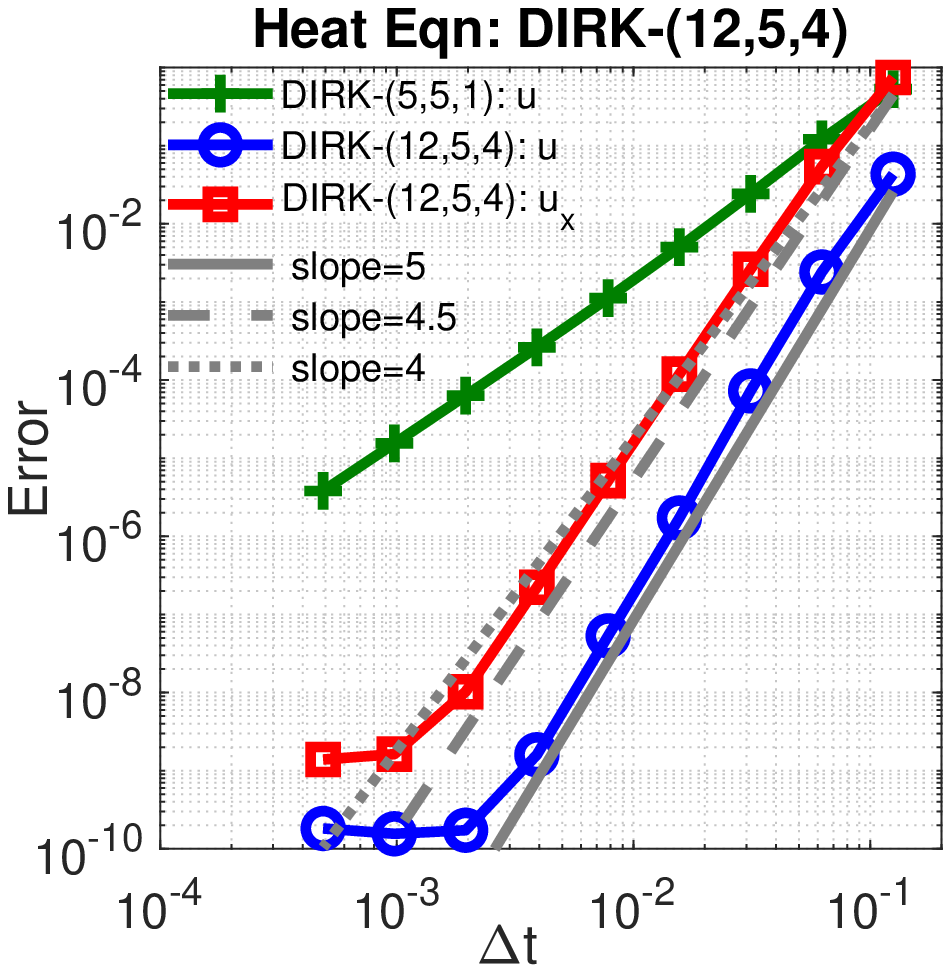}
	\end{minipage}
	\begin{minipage}[b]{.32\textwidth}
		\includegraphics[width=\textwidth]{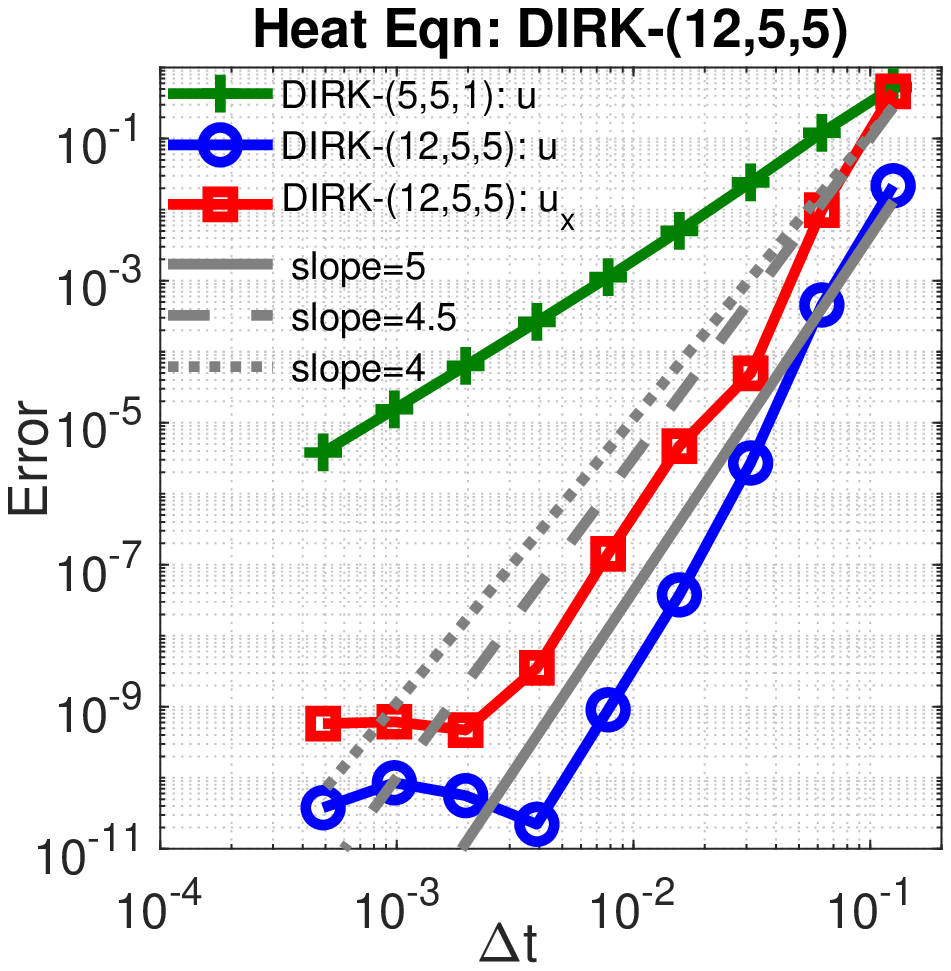}
	\end{minipage}
	\vspace{-.5em}
	\caption{Convergence ($u$ blue circles; $u_x$ red squares) for heat equation using DIRK-$(7,4,4)$: $4$th order DIRK scheme with WSO $4$ (left), DIRK-$(12,5,4)$: $5$th order DIRK scheme with WSO $4$ (middle), and DIRK-$(12,5,5)$: $5$th order DIRK scheme with WSO $5$ (right).}
	\label{HeatEqn}
\end{figure}

\subsection{Schr\"{o}dinger equation}
As an example of a dispersive problem we consider
\begin{equation*}
u_t = \frac{i \omega}{k^2} u_{xx}+f \ \ \text{for} \ (x,t)\in (0,1) \times (0,1.2], \quad u=g(x,t) \ \ \text{on} \ \{0,1\} \times (0,1.2]\;,
\end{equation*}
with the manufactured solution $u(x,t) = \exp\left(-(x-t)^2\right) \cos(10x) \sin(t)$, where $\omega = 2\pi$ and $k = 20$. As above, $u_{xx}$ is approximated by $4$th order centered differences on a fine grid with $10^4$ cells. The problem is solved up to final time $T=1.2$ via different RK schemes, and the convergence in $u$ and $u_x$ is evaluated.
Figure~\ref{fig:SchroedingeEquation} shows the results obtained with DIRK-$(7,4,4)$ (left), DIRK-$(12,5,4)$ (middle), and DIRK-$(12,5,5)$ (right), relative to DIRK-$(5,4,1)$ and DIRK-$(5,5,1)$ reference methods. Similar convergence results are observed as for the heat equation: full orders in $u$ and $u_x$ are recovered with DIRK-$(7,4,4)$ and DIRK-$(12,5,5)$, while $u_x$ loses a half order with DIRK-$(12,5,4)$.

\begin{figure}[htb]
	\begin{minipage}[b]{.32\textwidth}
		\includegraphics[width=\textwidth]{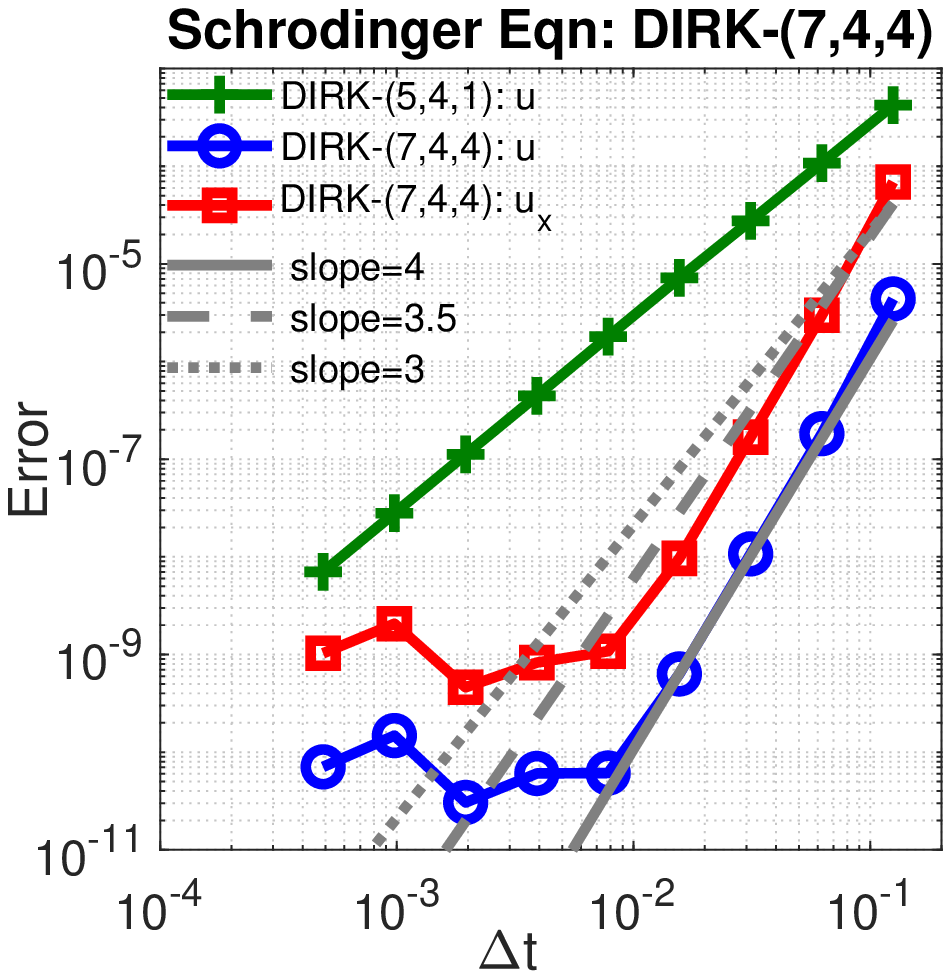}
	\end{minipage}
	\hfill
	\begin{minipage}[b]{.32\textwidth}
		\includegraphics[width=\textwidth]{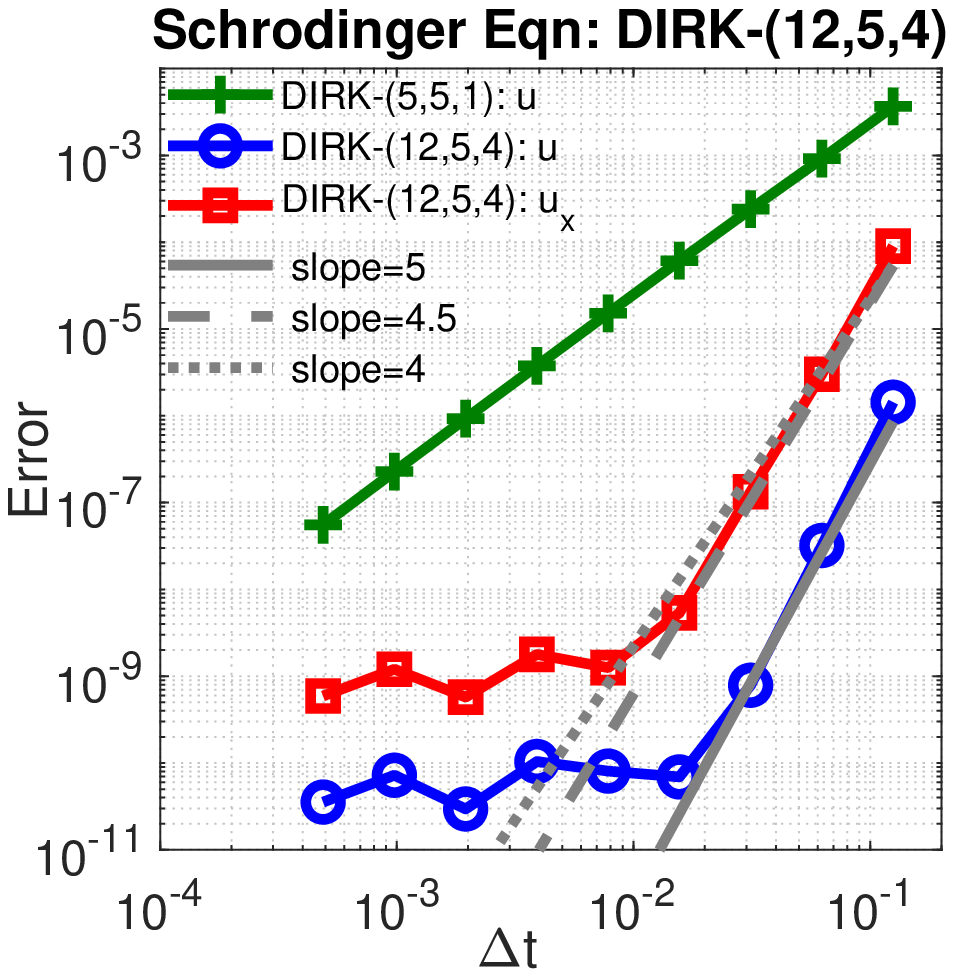}
	\end{minipage}
	\hfill
	\begin{minipage}[b]{.32\textwidth}
		\includegraphics[width=\textwidth]{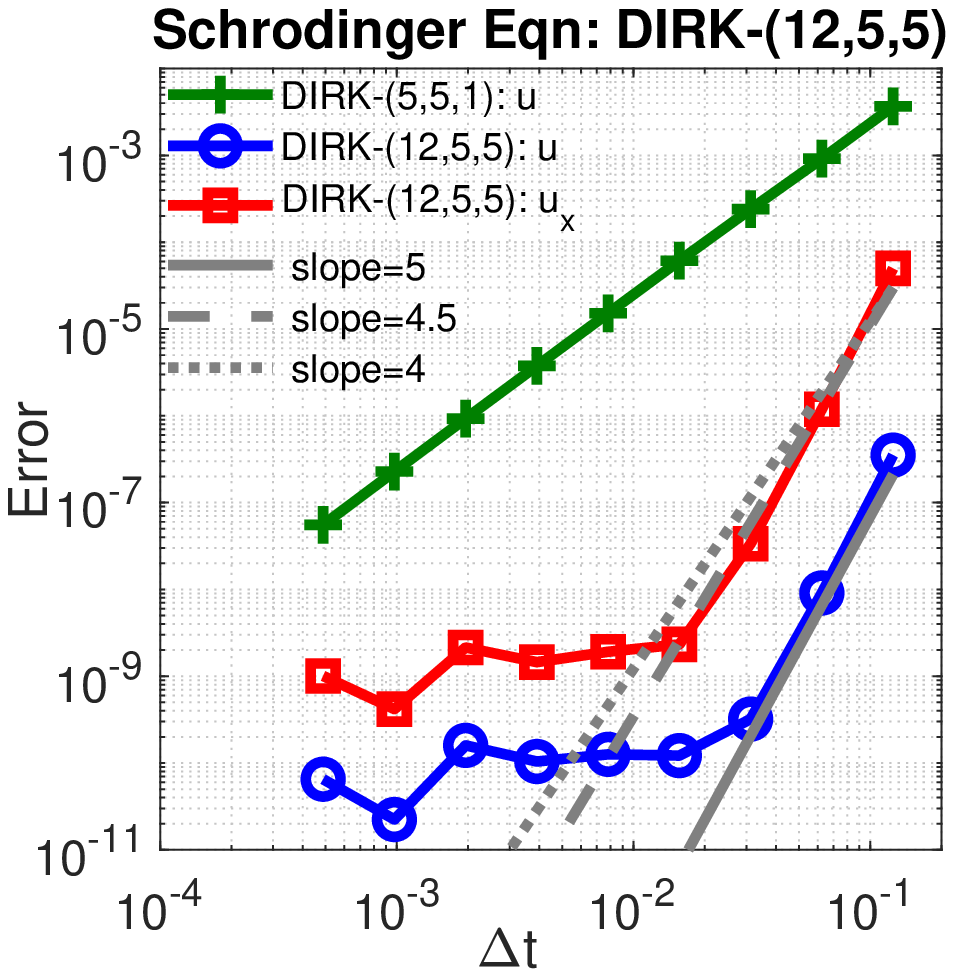}   
	\end{minipage}
	\vspace{-.5em}
	\caption{Convergence ($u$ blue circles; $u_x$ red squares) for Schr\"{o}dinger equation using DIRK-$(7,4,4)$: $4$th order DIRK scheme with WSO $4$ (left), DIRK-$(12,5,4)$: $5$th order DIRK scheme with WSO $4$ (middle), and DIRK-$(12,5,5)$: $5$th order DIRK scheme with WSO $5$ (right).}
	\label{fig:SchroedingeEquation}
\end{figure}

\subsection{Advection-diffusion equation}
This example demonstrates that DIRK schemes with high weak stage order avoid order reduction when applied to problems with physical boundary layers. We consider the $1$D linear advection-diffusion equation
\begin{equation*}
u_t+u_{x} = \nu u_{xx}+f(x,t) \ \ \text{for} \ (x,t)\in (0,1) \times (0,1] \;,
\end{equation*}
with the true solution $u(x,t) = \cos(5t)\sin(10x+10)$, Dirichlet b.c., and viscosity $\nu = 10^{-3}$. The advection term dominates, and the outflow boundary condition at $x=1$ leads to a physical boundary layer of width $\mathcal{O}(\nu)$. All spatial derivatives are approximated via $4\text{th}$-order centered differences on a grid with $10^4$ cells, and errors are evaluated at $T=1$. The results shown in Figure~\ref{fig:LinAdvDiffEqn} exhibit the expected convergence in $u$ and $u_x$ for DIRK-$(7,4,4)$ (left), DIRK-$(12,5,4)$ (middle), and DIRK-$(12,5,5)$ (right). In particular, the results confirm that physical boundary layers do not interfere with the schemes' remedy of order reduction.

\begin{figure}[htb]
	\begin{minipage}[b]{.32\textwidth}
		\includegraphics[width=\textwidth]{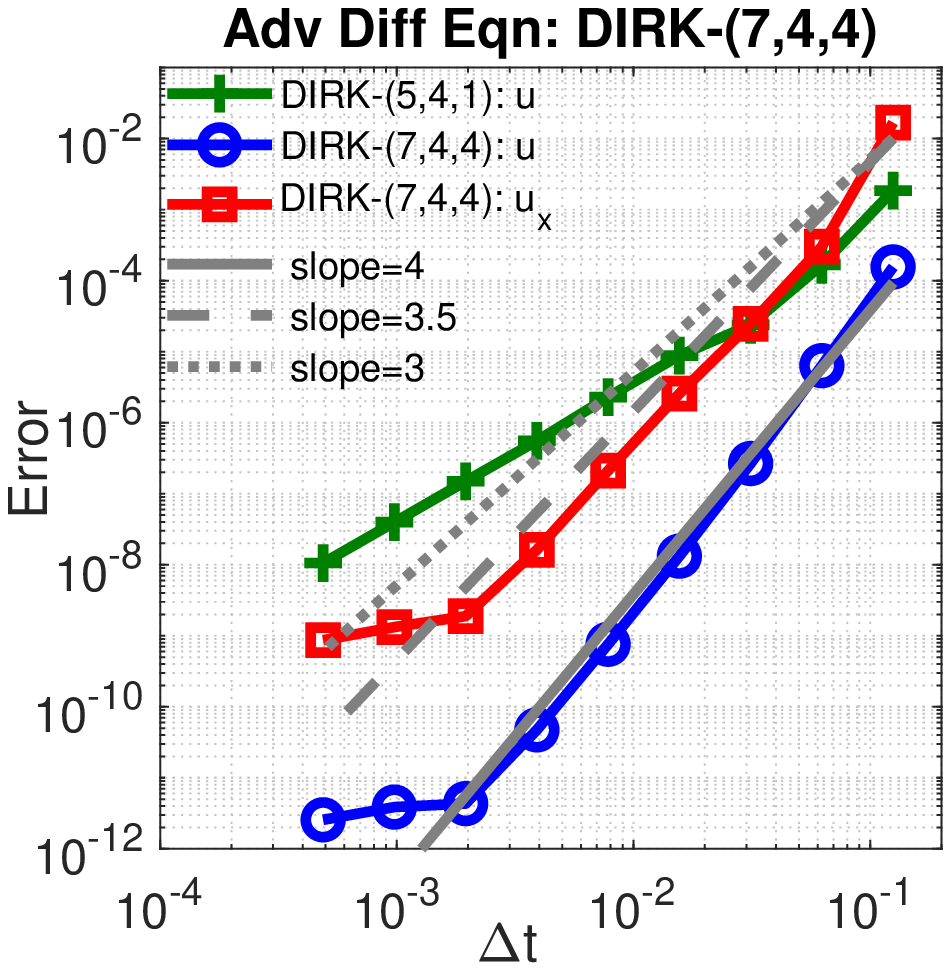}
	\end{minipage}
	\hfill
	\begin{minipage}[b]{.32\textwidth}
		\includegraphics[width=\textwidth]{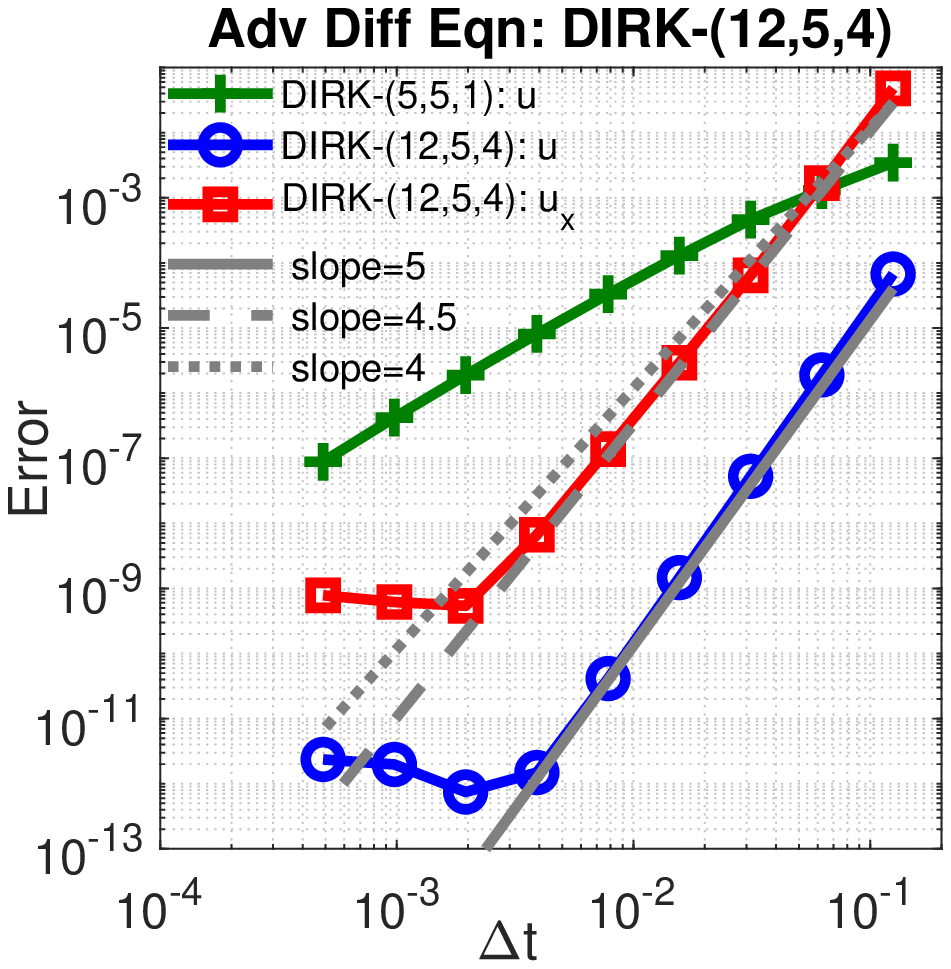}
	\end{minipage}
	\hfill
	\begin{minipage}[b]{.32\textwidth}
		\includegraphics[width=\textwidth]{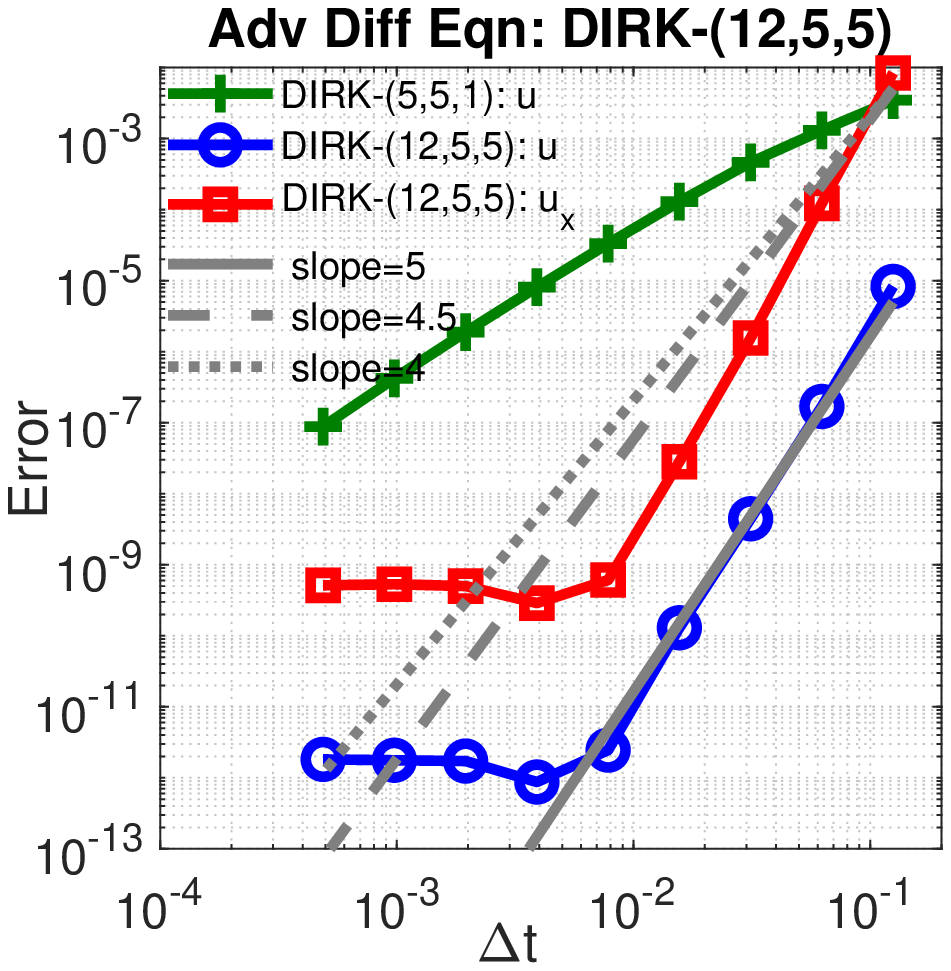}
	\end{minipage}
	\vspace{-.5em}
	\caption{Convergence ($u$ blue circles; $u_x$ red squares) for the advection-diffusion equation using DIRK-$(7,4,4)$: $4$th order DIRK scheme with WSO $4$ (left), DIRK-$(12,5,4)$: $5$th order DIRK scheme with WSO $4$ (middle), and DIRK-$(12,5,5)$: $5$th order DIRK scheme with WSO $5$ (right).}
	\label{fig:LinAdvDiffEqn}
\end{figure}

\subsection{Linear advection equation}\label{subsec:linadveq}
To illustrate order reduction, and its remedy, in problems with only first-order spatial derivatives, we consider
\begin{equation*}
u_t+u_{x} = 0 \ \ \text{for} \ (x,t)\in (0,1) \times (0,1]\;,
\end{equation*}
with Dirichlet b.c.~at $x = 0$, final time $T=1$, and the true solution a traveling wave $u(x,t) = \sin(2\pi(x-t))$. Again, $4$th-order centered differences with $10^4$ cells are used to approximate $\partial_x$.
For this first-order problem, the numerical boundary layer due to order reduction now is of thickness $\mathcal{O}(\Delta t)$, hence we expect a loss of a full order in $u_x$ when $q<p$. This is demonstrated in Figure~\ref{fig:LinAdvEqn}: DIRK-$(12,5,4)$ recovers $5$th order in $u$ and $4$th order in $u_x$. Moreover, DIRK-$(7,4,4)$ and DIRK-$(12,5,5)$ recover their full orders of convergence for both $u$ and $u_{x}$. It is interesting to note that for this specific test problem, the reference schemes with $q=1$ turn out to exhibit third-order convergence, instead of the expected second order. We do not have an explanation for this interesting behavior; however, note that this is not in contradiction to any of the theory.

\begin{figure}[htb]
	\begin{minipage}[b]{.32\textwidth}
		\includegraphics[width=\textwidth]{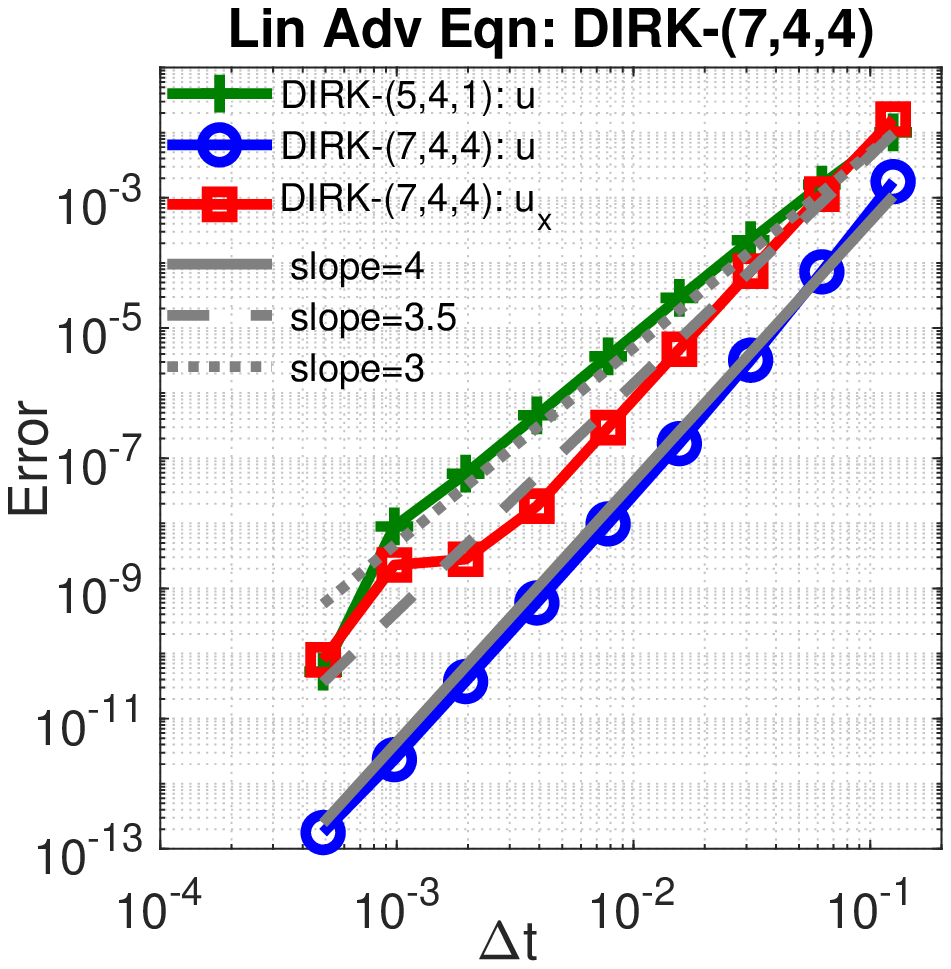}
	\end{minipage}
	\hfill
	\begin{minipage}[b]{.32\textwidth}
		\includegraphics[width=\textwidth]{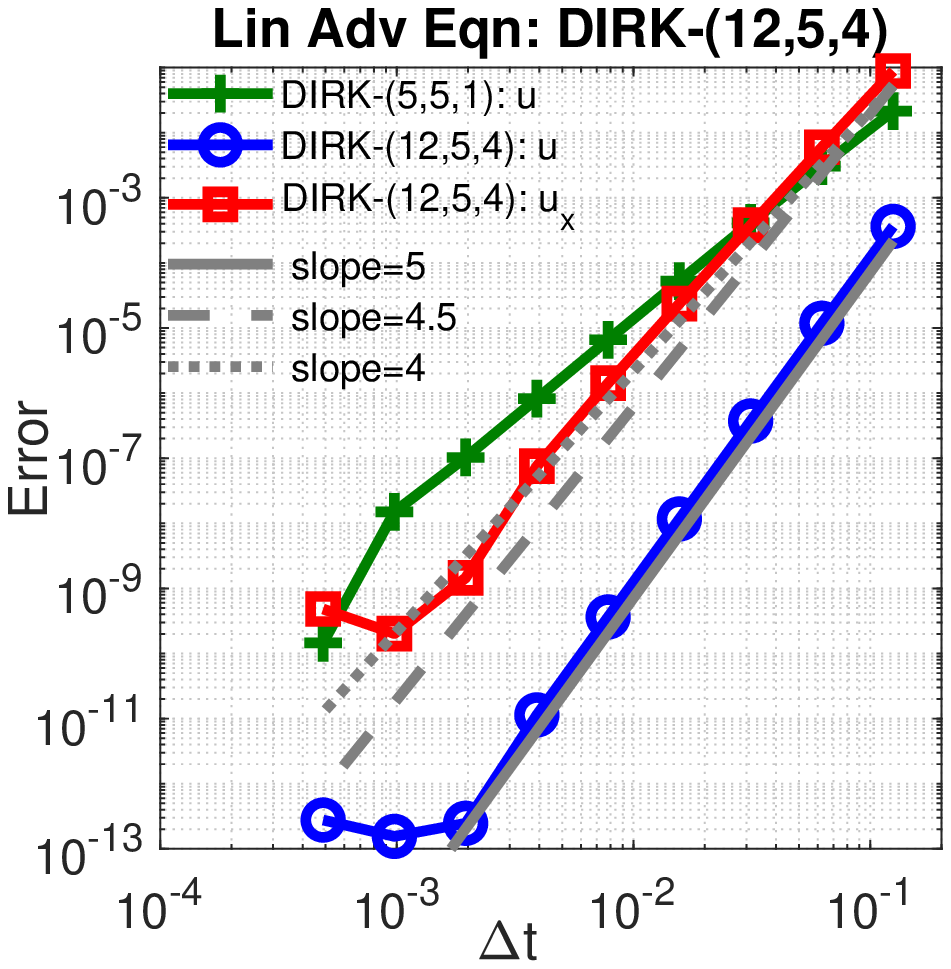}
	\end{minipage}
	\hfill
	\begin{minipage}[b]{.32\textwidth}
		\includegraphics[width=\textwidth]{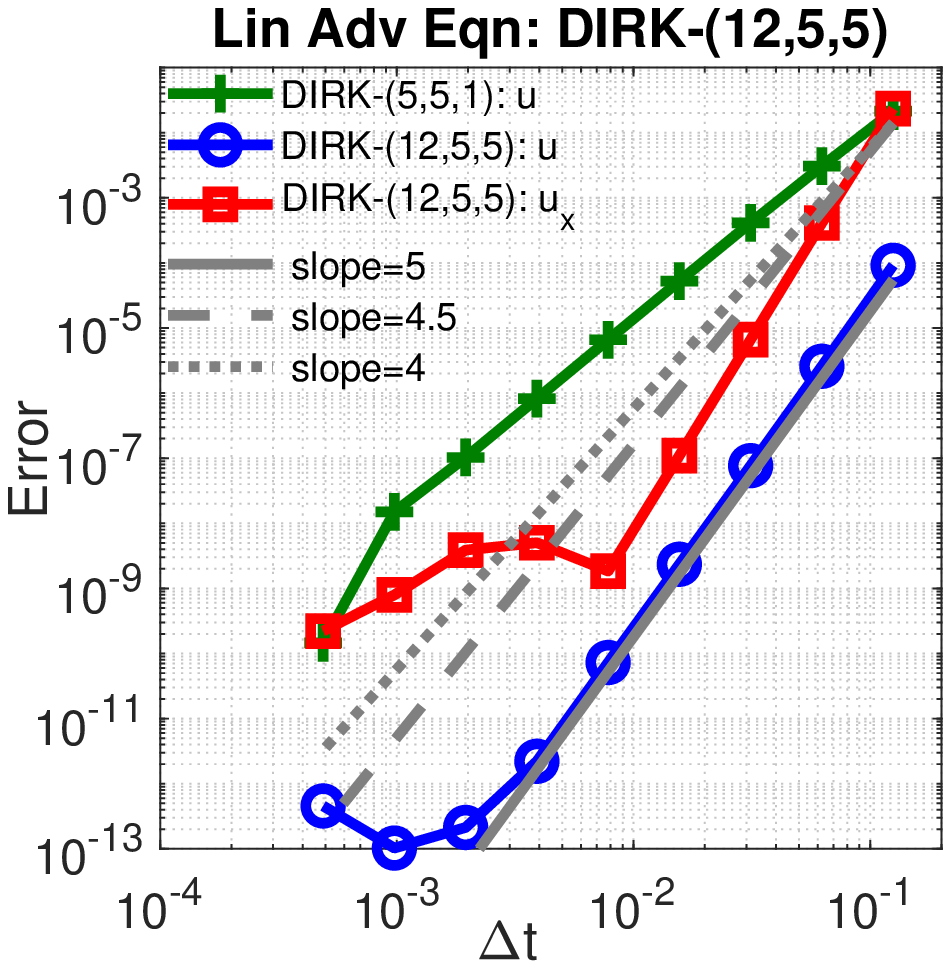}
	\end{minipage}
	\vspace{-.5em}
	\caption{Convergence ($u$ blue circles; $u_x$ red squares) for the linear advection equation using DIRK-$(7,4,4)$: $4$th order DIRK scheme with WSO $4$ (left), DIRK-$(12,5,4)$: $5$th order DIRK scheme with WSO $4$ (middle), and DIRK-$(12,5,5)$: $5$th order DIRK scheme with WSO $5$ (right).}
	\label{fig:LinAdvEqn}
\end{figure}

\subsection{Heat equation with spatially varying coefficient}
\label{subsec:results_heateqn_varying_x}
The examples above are restricted to differential operators with constant coefficients. To demonstrate that our schemes remedy order reduction for more general problems, we consider the heat equation
\begin{equation*}
	u_t=\left(\kappa(x)u_{x}\right)_{x}+f \ \ \text{for} \ (x,t)\in (0,1) \times (0,1], \quad u=g \ \ \text{on} \ \{0,1\} \times (0,1]\;,
\end{equation*}
with spatially varying diffusion coefficient $\kappa(x) = \cos(x+0.1)$. The forcing $f(x,t)$, the b.c., and the i.c.\ are chosen so that the true solution is $u(x,t) = \cos(20t)\sin(10x+10)$. We used $6$th-order centered differences with $10^3$ cells to approximate the spatial derivatives and the problem is solved till $T=1$.
Figure~\ref{fig:SpaceDepnVarCoeffHeatEqn} confirms that the high WSO schemes recover the expected convergence orders, just as they did for the constant-coefficient heat equation.

\begin{figure}[htb]
	\begin{minipage}[b]{.32\textwidth}
		\includegraphics[width=\textwidth]{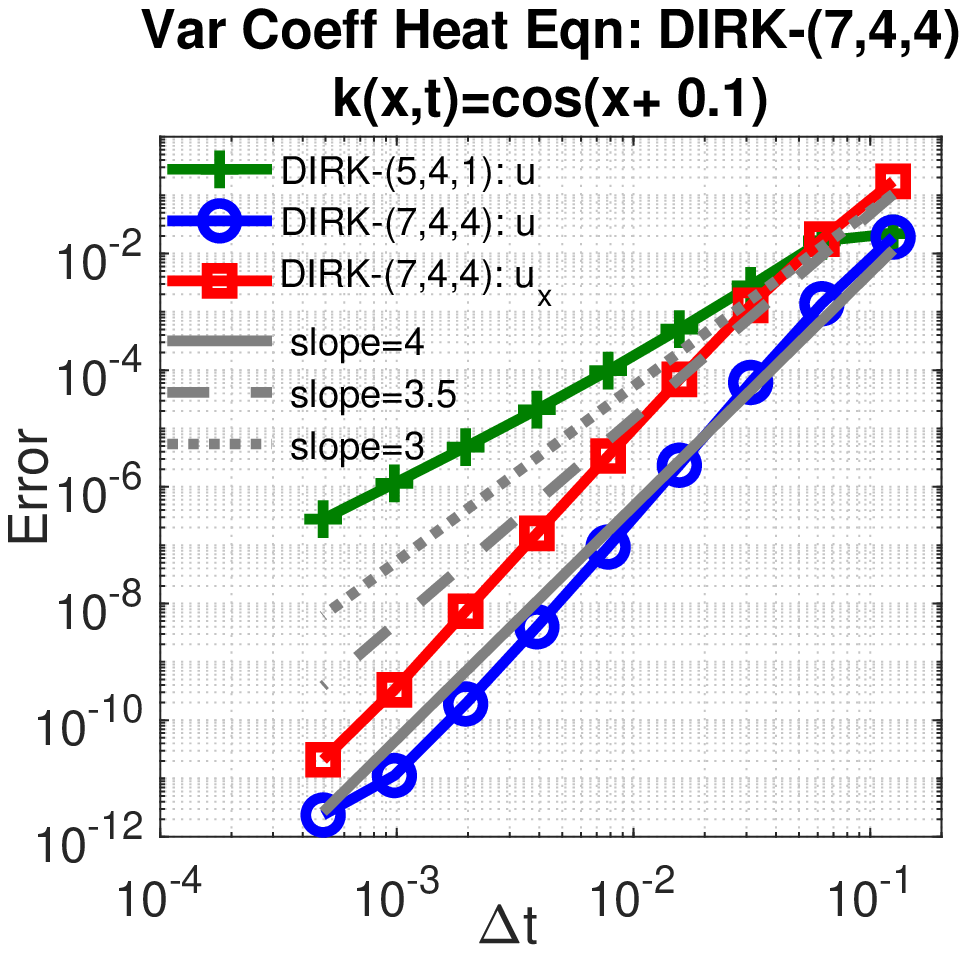}
	\end{minipage}
	\begin{minipage}[b]{.32\textwidth}
		\includegraphics[width=\textwidth]{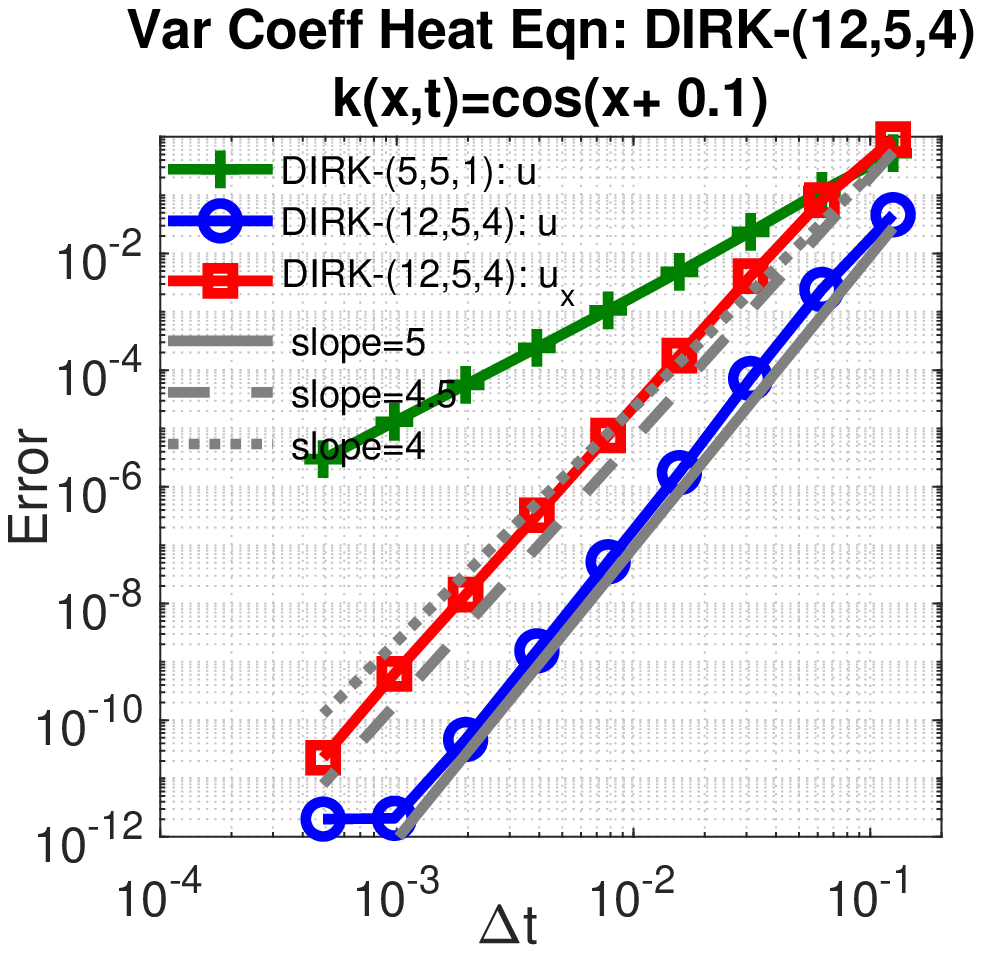}
	\end{minipage}
	\begin{minipage}[b]{.32\textwidth}
		\includegraphics[width=\textwidth]{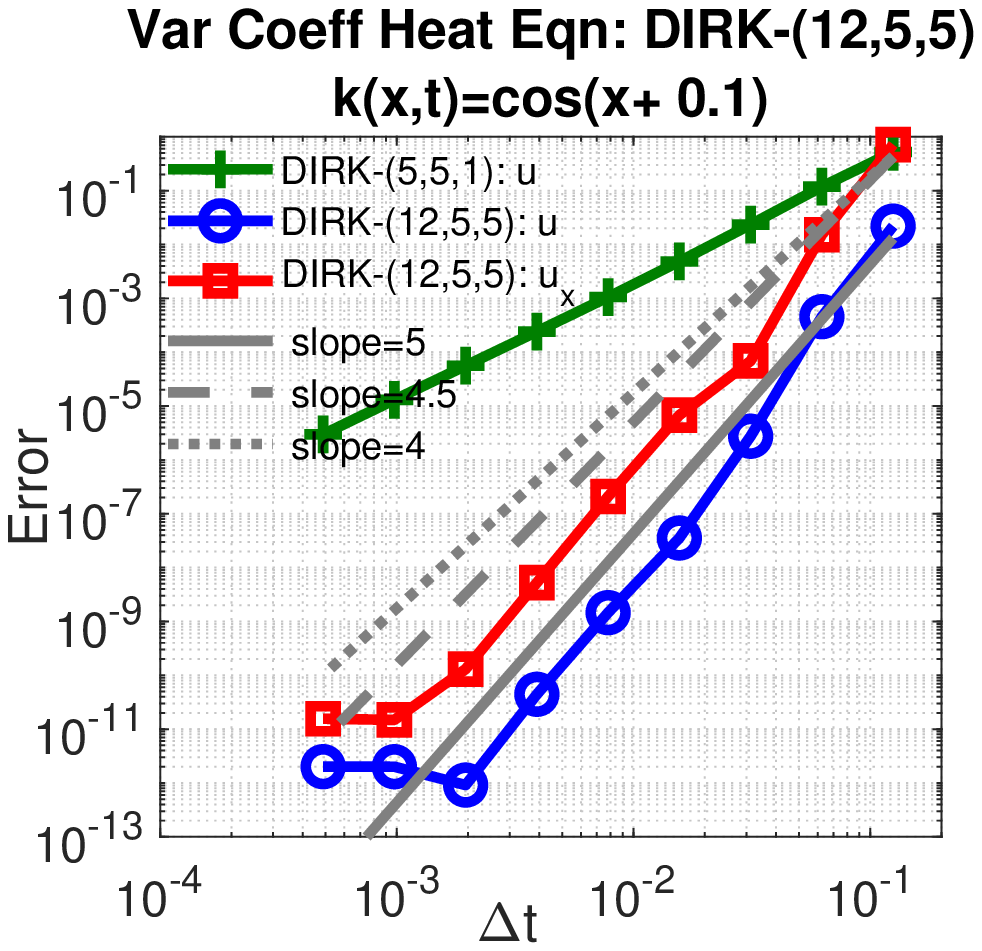}
	\end{minipage}
	\vspace{-.5em}
	\caption{Convergence ($u$ blue circles; $u_x$ red squares) for heat equation with spatially varying diffusion coefficient $\kappa(x) = \cos(x+0.1)$ using DIRK-$(7,4,4)$: $4$th order DIRK scheme with WSO $4$ (left), DIRK-$(12,5,4)$: $5$th order DIRK scheme with WSO $4$ (middle), and DIRK-$(12,5,5)$: $5$th order DIRK scheme with WSO $5$ (right).}
	\label{fig:SpaceDepnVarCoeffHeatEqn}
\end{figure}

\subsection{An equation with a fourth-order spatial derivative}
To demonstrate that our schemes remove order reduction for PDEs with more than one boundary condition, we consider 
\begin{equation}\label{eq:biharmonic}
u_t=-u_{xxxx} +f\ \text{for} \ (x,t)\in (0,1) \times (0,1] \;,
\end{equation}
with both ``Dirichlet'' and ``Neumann'' boundary conditions on each side, i.e., $u(0) = g_0(t)$, $u(1) = g_1(t)$, $u_x(0) = h_0(t)$, $u_x(1) = h_1(t)$, and the forcing $f$ such that the manufactured solution is $u(x,t) = \cos(15t)$. The final time is $T = 1$. In this equation, the $4$th order spatial derivative is approximated by a $2$nd-order centered finite difference on a fine grid of $10^4$ cells.
\begin{figure}[htb]
	\begin{minipage}[b]{.32\textwidth}
		\includegraphics[width=\textwidth]{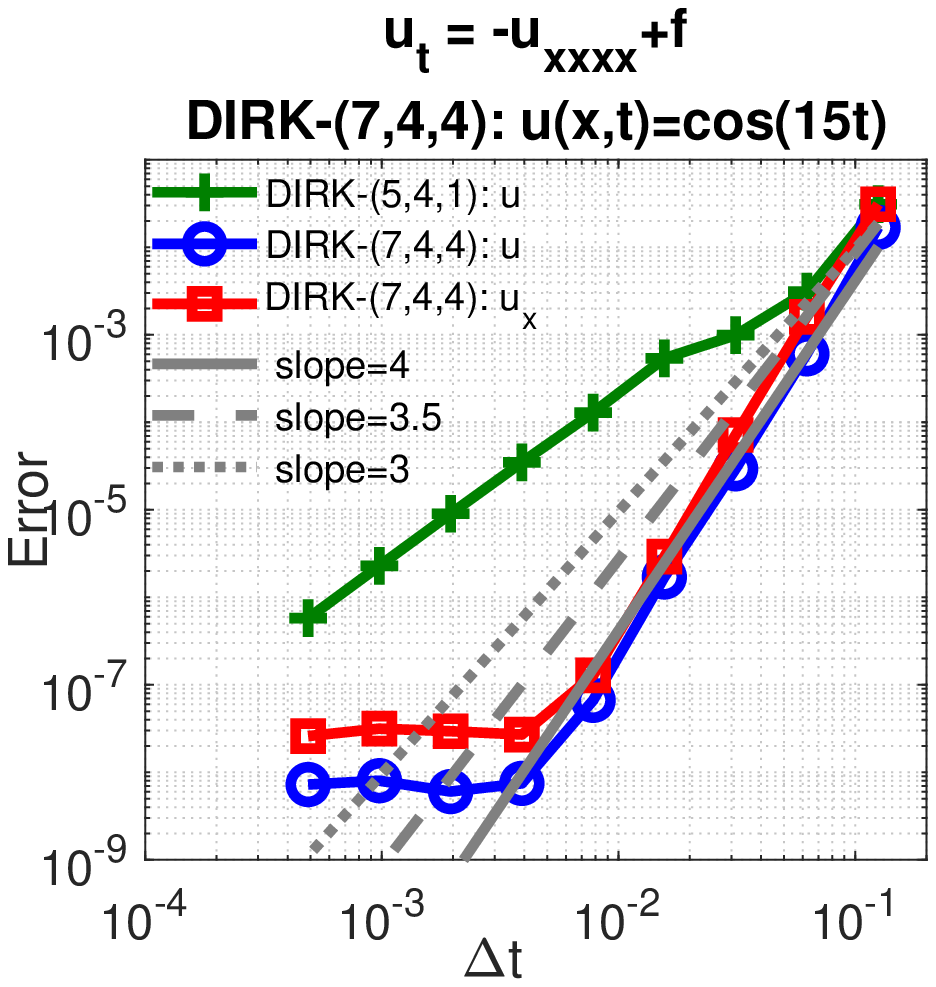}
	\end{minipage}
	\hfill
	\begin{minipage}[b]{.32\textwidth}
	\includegraphics[width=\textwidth]{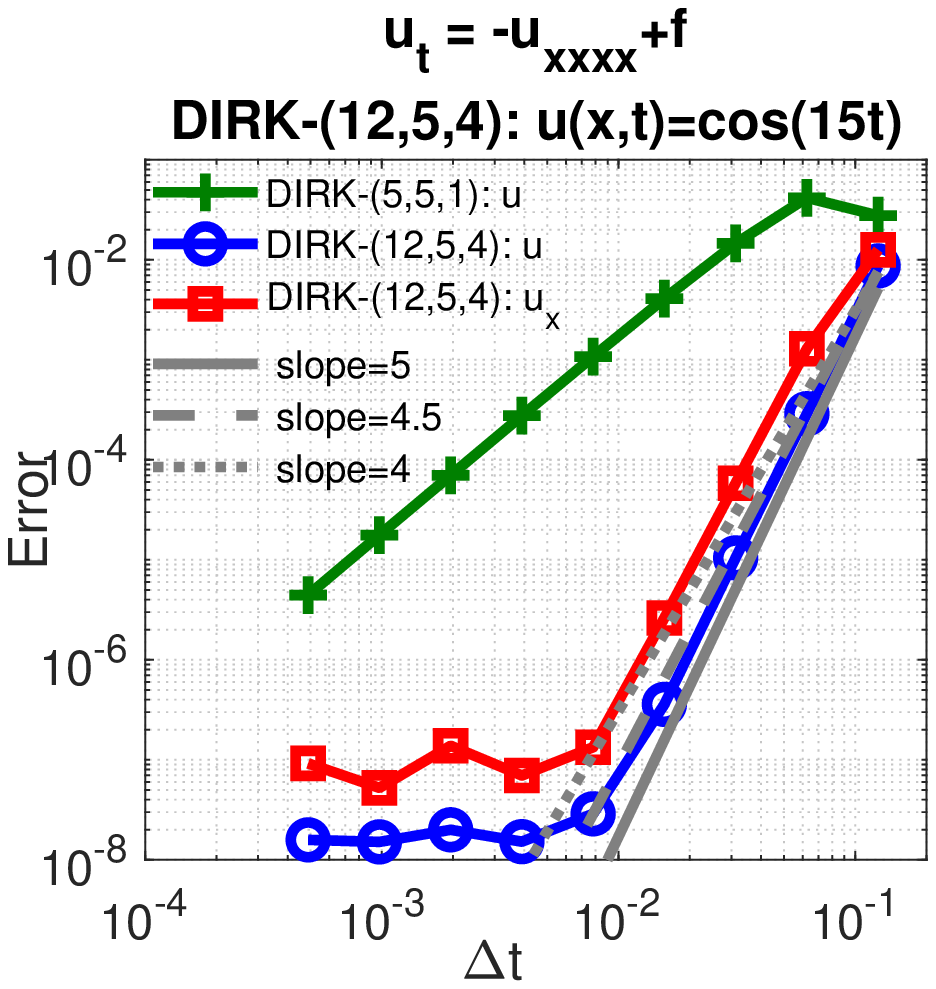}
	\end{minipage}
	\hfill
	\begin{minipage}[b]{.32\textwidth}
    \includegraphics[width=\textwidth]{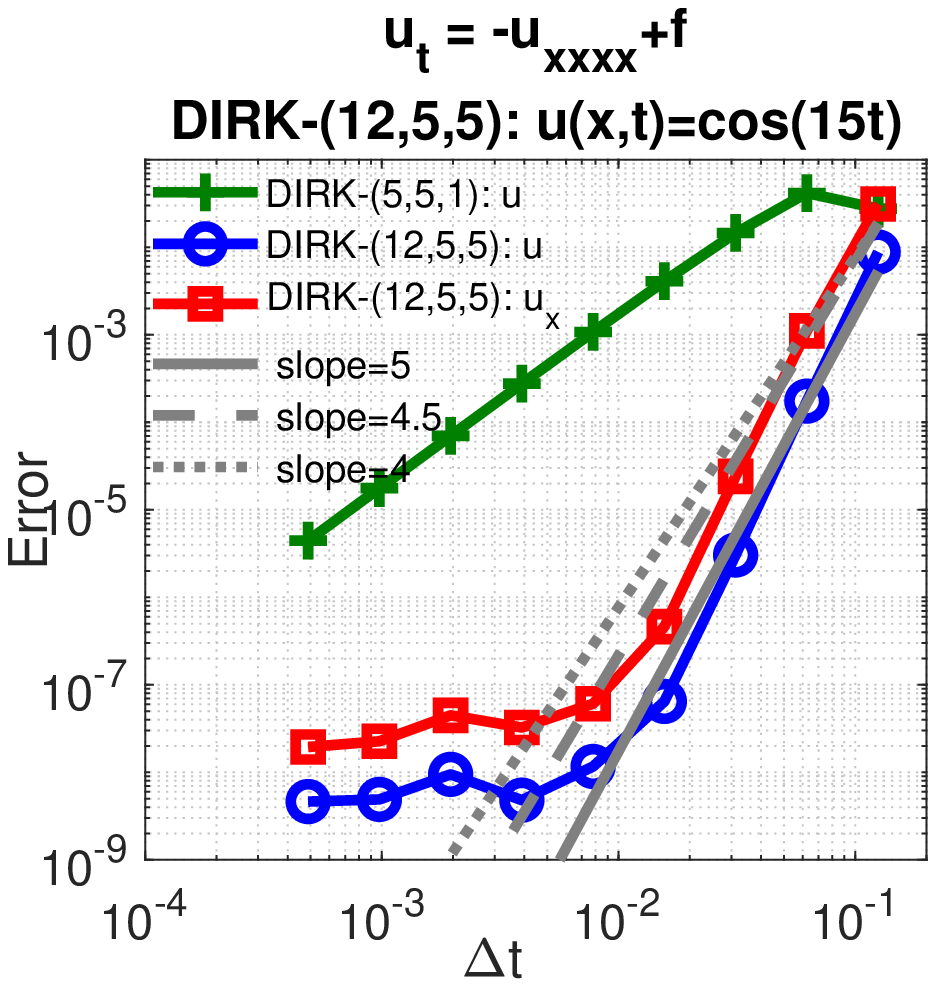}
	\end{minipage}
	\caption{Error convergence ($u$ blue; $u_x$ red) for equation \eqref{eq:biharmonic} using DIRK-$(7,4,4)$: $4$th order DIRK scheme with WSO $4$ (left), DIRK-$(12,5,4)$: $5$th order DIRK scheme with WSO $4$ (middle), and DIRK-$(12,5,5)$: $5$th order DIRK scheme with WSO $5$ (right).}
	\label{fig:BiharmonicEqn}
\end{figure}

The convergence results, obtained with the new high WSO DIRK schemes, as well as the reference WSO 1 schemes, are shown in Figure~\ref{fig:BiharmonicEqn}. The schemes with $q = p$ recover the full order of convergence for both $u$ and $u_x$. The time-stepping schemes produce numerical boundary layers whose width scales like $\mathcal{O}(\Delta t^{\frac{1}{4}})$, leading to $\frac{1}{4}$ order loss per derivative for the schemes that have $q < p$.

\subsection{Two-dimensional linear advection-diffusion equation}
In principle, the presence of corners (non-smooth domain boundaries) may be an additional source of error that could lead to order reduction. Here we examine a two-dimensional PDE problem in a square domain and demonstrate that our schemes remedy order reduction also in this setting. We consider the advection-diffusion equation
\begin{equation*}
u_t + u_x + u_y = \nu(u_{xx} + u_{yy}) + f(x,y), \ (x,y) \in [-1,1]^2 \;,
\end{equation*}
with $\nu = 0.1$, the forcing $f$, the boundary conditions and initial condition chosen such that the manufactured solution is $u(x,y,t) = \exp(-\frac{\pi^2}{8}t)\sin(\pi x+\frac{\pi}{4})\sin(\pi y+\frac{\pi}{4})$. 

\begin{figure}[htb]
	\begin{minipage}[b]{.32\textwidth}
		\includegraphics[width=\textwidth]{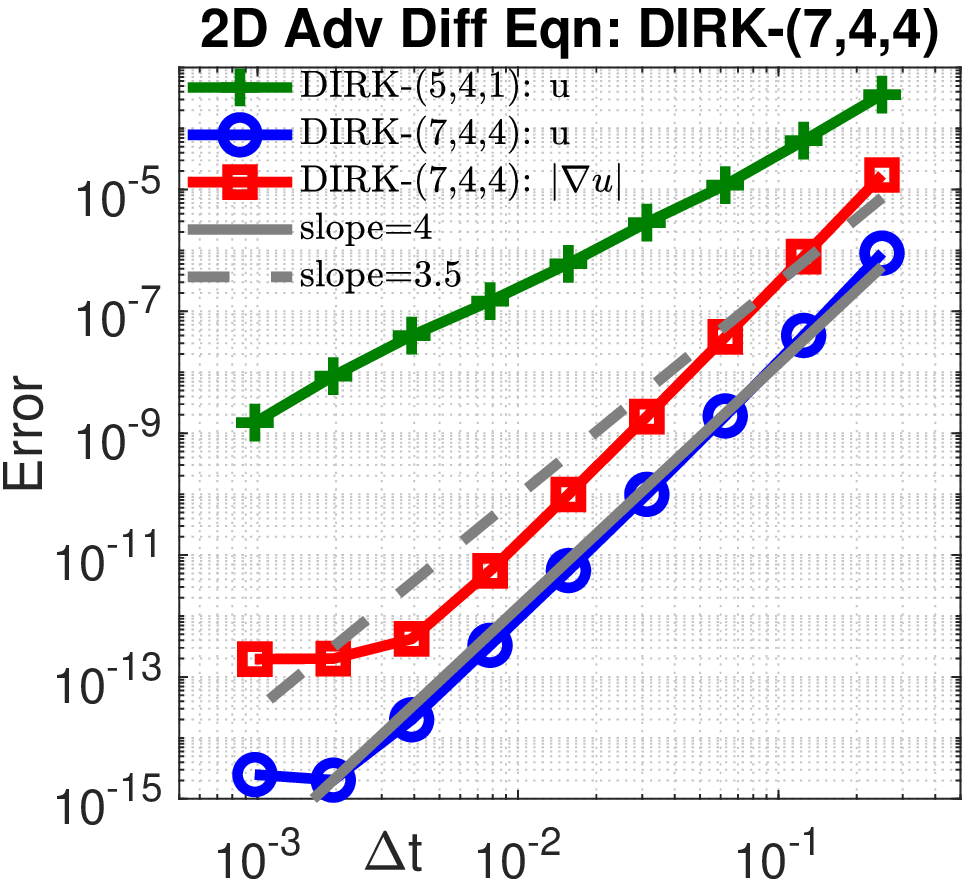}
	\end{minipage}
	\hfill
	\begin{minipage}[b]{.32\textwidth}
	\includegraphics[width=\textwidth]{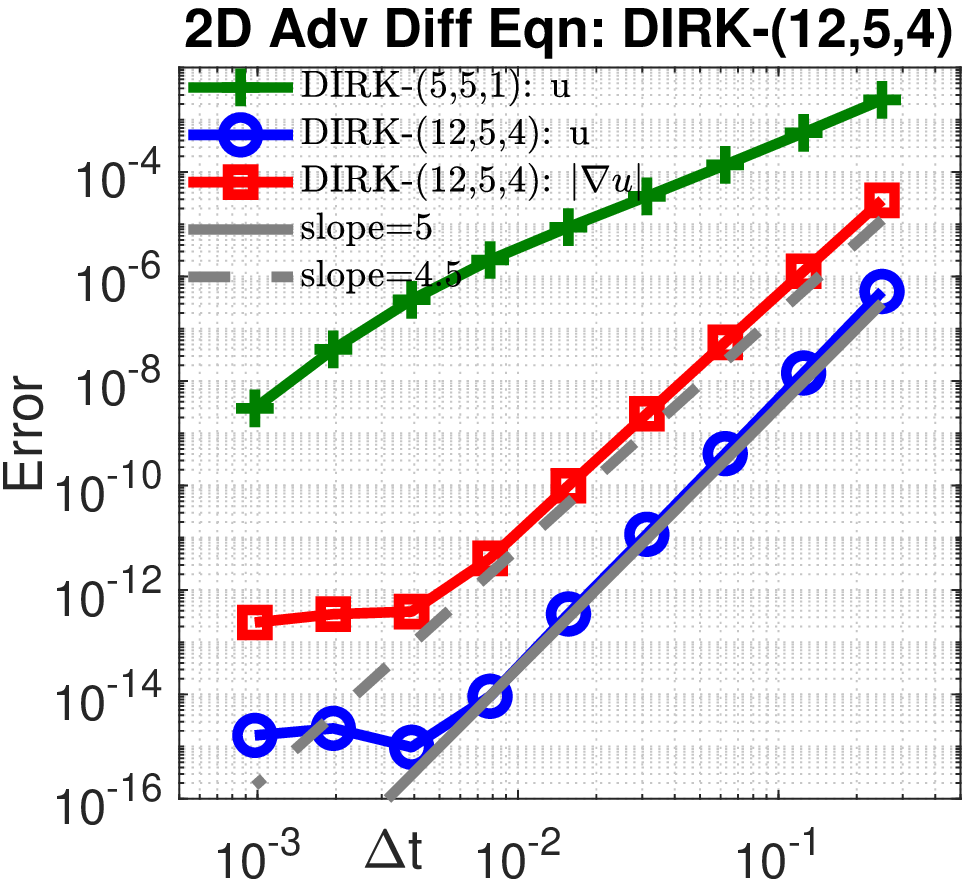}
	\end{minipage}
	\hfill
	\begin{minipage}[b]{.32\textwidth}
    \includegraphics[width=\textwidth]{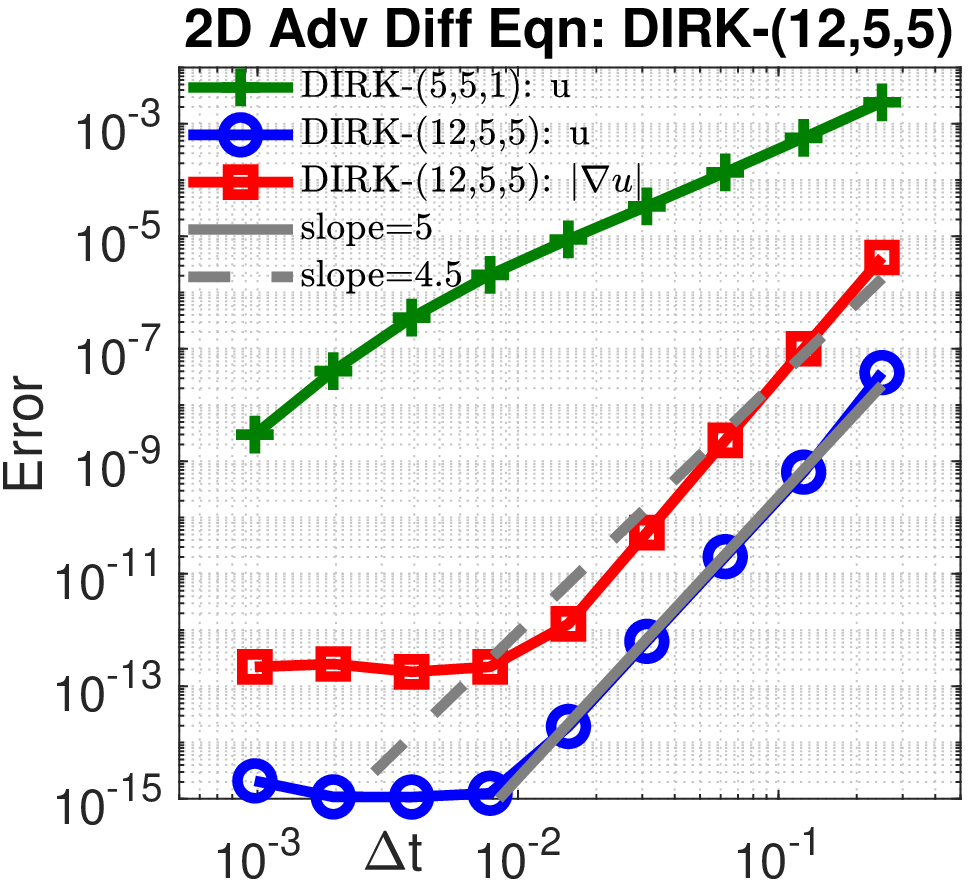}
	\end{minipage}
	\caption{Error convergence ($u$ blue; $|\nabla u|$ red) for a two-dimensional linear advection equation using DIRK-$(7,4,4)$: $4$th order DIRK scheme with WSO $4$ (left), DIRK-$(12,5,4)$: $5$th order DIRK scheme with WSO $4$ (middle), and DIRK-$(12,5,5)$: $5$th order DIRK scheme with WSO $5$ (right).}
	\label{fig:2dAdvDiffEqn}
\end{figure}

To ensure the spatial error is negligible, we use a spectral method on a $2$D tensor-product grid with $30$ Chebyshev points in each direction. We solve the problem up to $T=1$. Errors are plotted in Figure~\ref{fig:2dAdvDiffEqn}. The schemes with high WSO successfully avoid order reduction for this problem.

\section{Numerical Results: Time-Dependent Linear and Nonlinear Operators}
\label{sec:numerical_result_more_general_problems}
The weak stage order conditions (see \S\ref{sec:OR_phenomenon}) are derived based on a linear problem in which the coefficient of the linear term is time-independent. Here we explore the question whether these conditions are also sufficient to alleviate order reduction for more general problems.

\subsection{Heat equation with temporally varying coefficient}
\label{subsec:he_tdep_coeff}
We revisit the variable-coefficient heat equation of \S\ref{subsec:results_heateqn_varying_x}, but now allowing $\kappa$ to vary also in time:
\begin{equation*}
u_t = \left(\kappa(x,t)u_{x}\right)_{x}+f \ \ \text{for} \ (x,t)\in (0,1) \times (0,1], \quad u=g \ \ \text{on} \ \{0,1\} \times (0,1]\;.
\end{equation*}
We consider two different diffusion coefficient functions, one of which varies slowly in time, $\kappa(x,t) = \cos(0.1t+0.2)$, and another that oscillates rapidly in time, $\kappa(x,t) = 1+0.5\cos(30t+0.1)$. In both cases, the spatial derivatives are approximated using $6$th-order centered differences with $10^3$ cells, and the errors are evaluated at time $T=1$.
Figure~\ref{fig:SlowTimeDepnVarCoeffHeatEqn} shows that all high WSO schemes practically alleviate order reduction for the slowly-varying coefficient case. In contrast, for the rapidly-varying coefficient case, the schemes suffer from order reduction and do not produce clean high-order convergence results. That being said, the new high WSO schemes do turn out to yield smaller errors than the WSO-1 reference methods.

\begin{figure}[htb]
	\begin{minipage}[b]{.32\textwidth}
		\includegraphics[width=\textwidth]{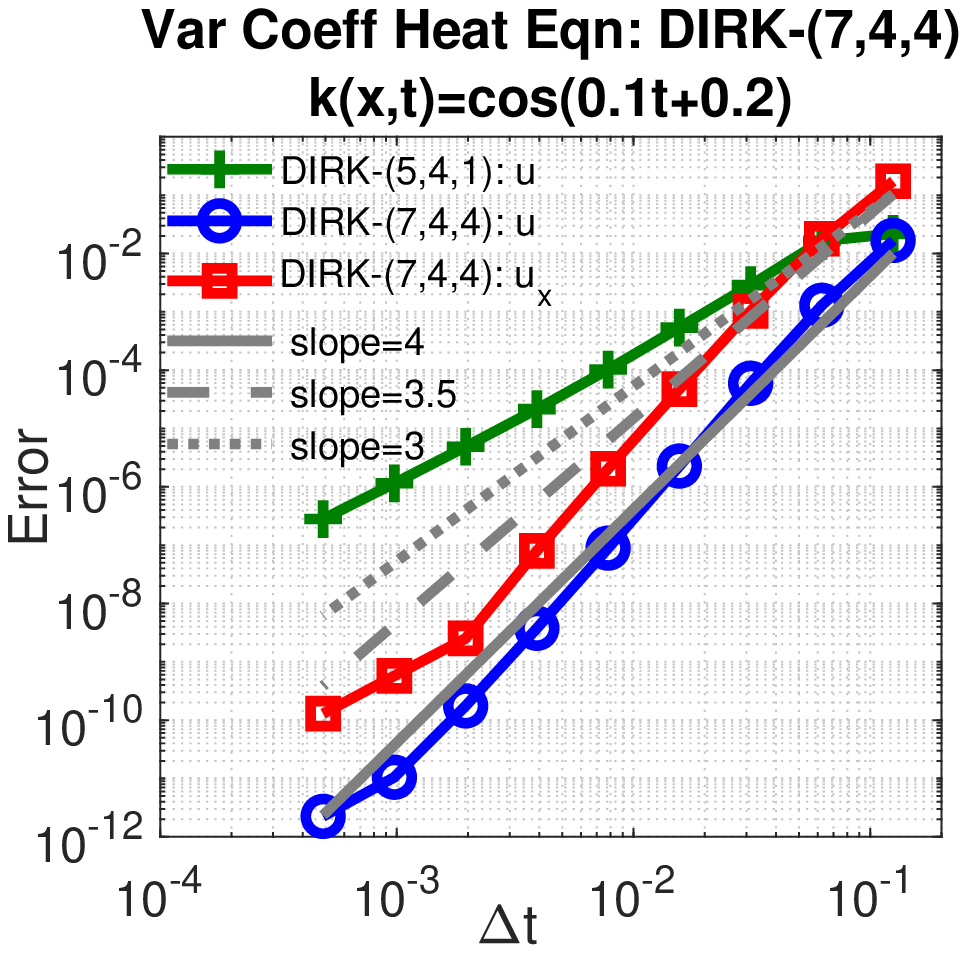}
	\end{minipage}
    \hfill
	\begin{minipage}[b]{.32\textwidth}
		\includegraphics[width=\textwidth]{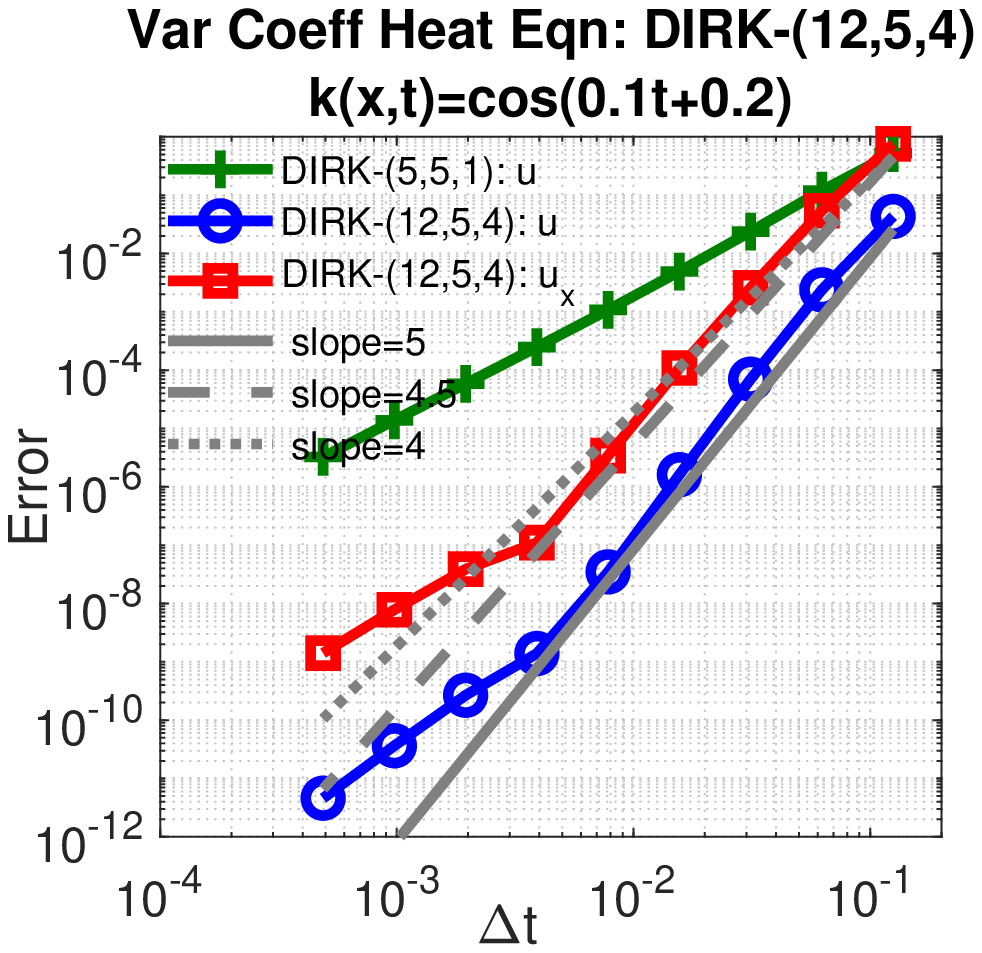}
	\end{minipage}
    \hfill
	\begin{minipage}[b]{.32\textwidth}
		\includegraphics[width=\textwidth]{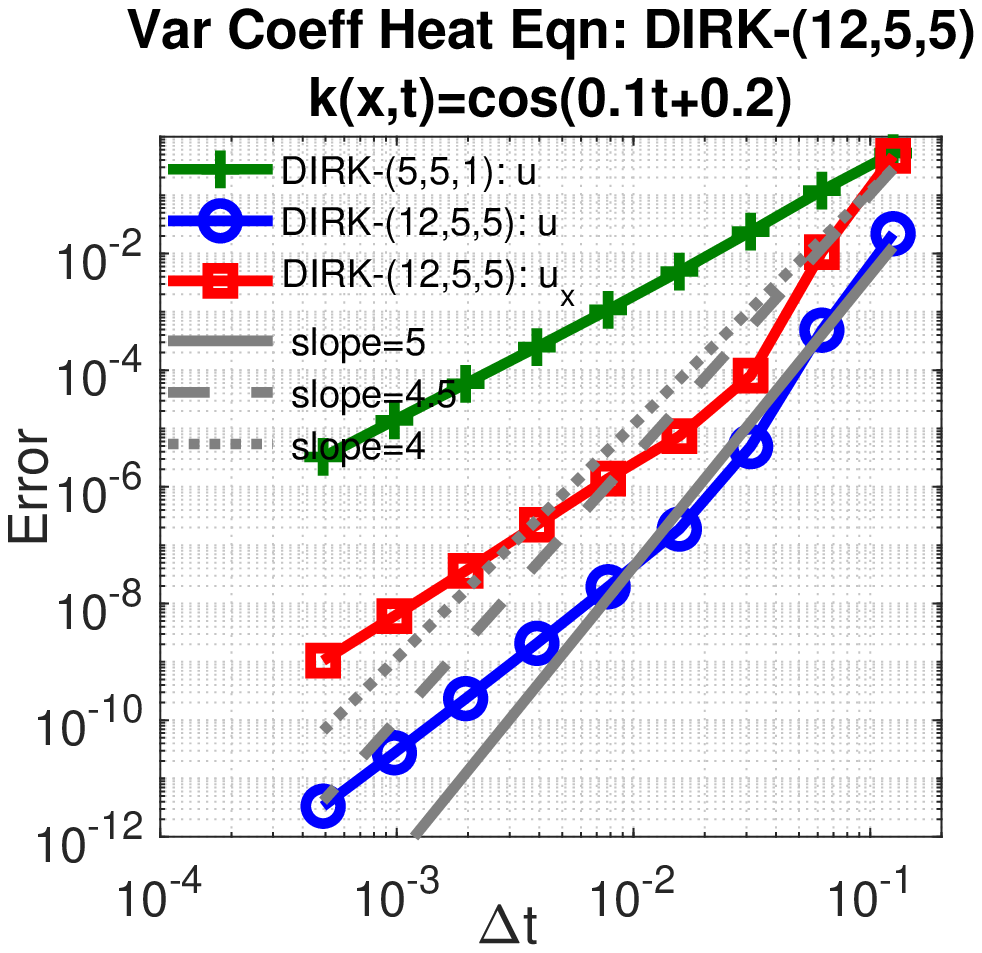}
	\end{minipage}
	
	\vspace{.5em}
	\begin{minipage}[h]{.32\textwidth}
	    \includegraphics[width=\textwidth]{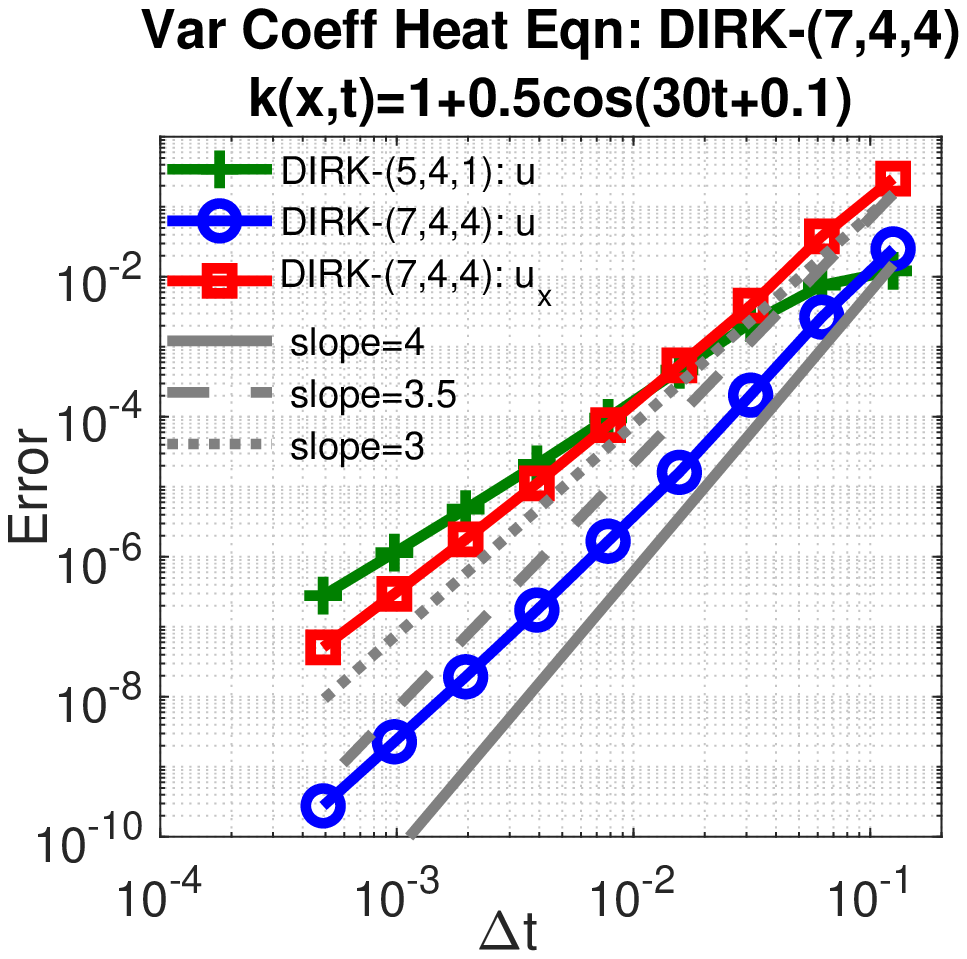}
    \end{minipage}
    \hfill
    \begin{minipage}[h]{.32\textwidth}
    	\includegraphics[width=\textwidth]{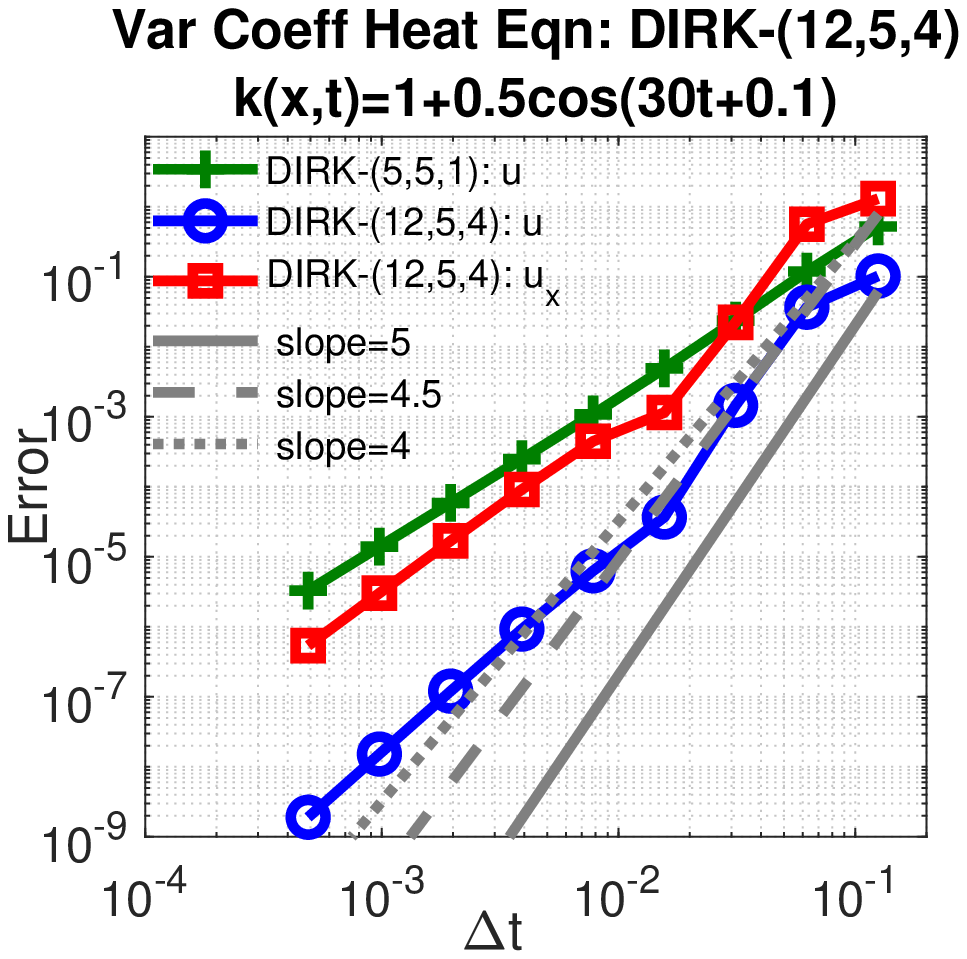}
    \end{minipage}
    \hfill
    \begin{minipage}[h]{.32\textwidth}
    	\includegraphics[width=\textwidth]{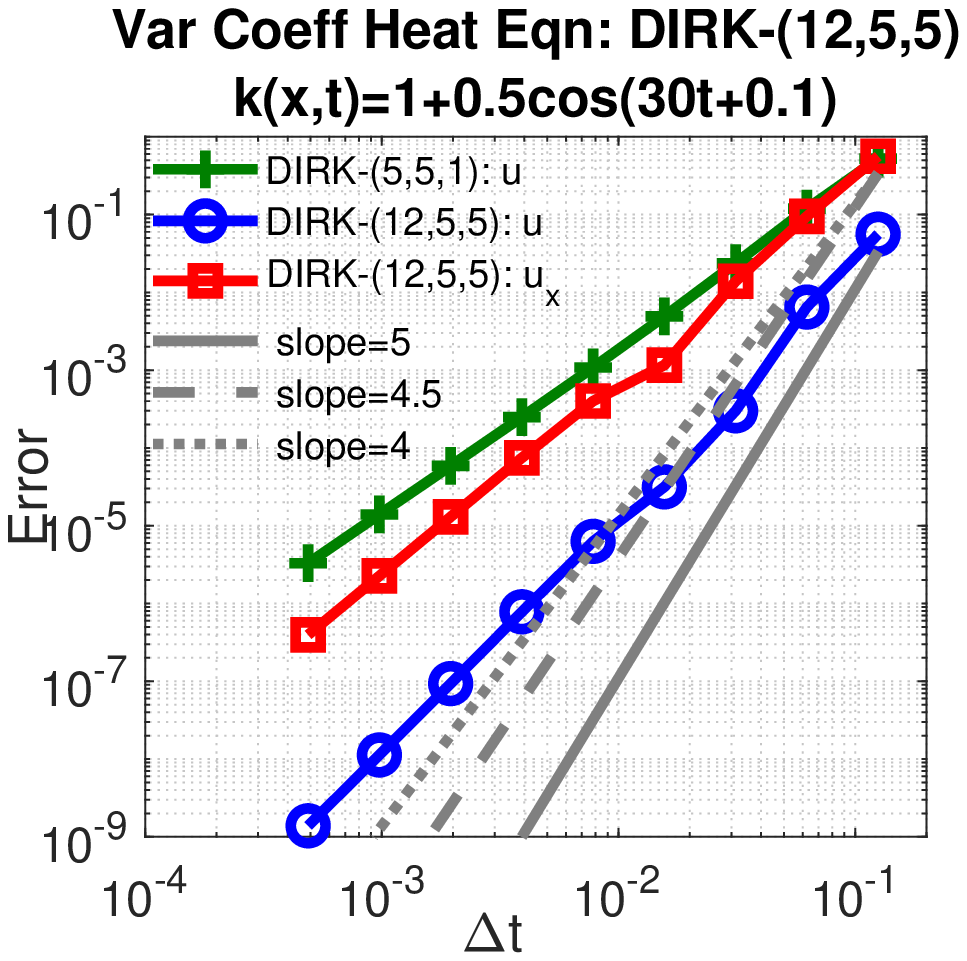}
    \end{minipage}
	\vspace{-.5em}
	\caption{Convergence ($u$ blue circles; $u_x$ red squares) for heat equation with diffusion coefficient $\kappa(x,t) = \cos(0.1t+0.2)$ (top) and $\kappa(x,t) = 1+0.5\cos(20t)$ (bottom) using DIRK-$(7,4,4)$: $4$th order DIRK scheme with WSO $4$ (left), DIRK-$(12,5,4)$: $5$th order DIRK scheme with WSO $4$ (middle), and DIRK-$(12,5,5)$: $5$th order DIRK scheme with WSO $5$ (right).}
	\label{fig:SlowTimeDepnVarCoeffHeatEqn}
\end{figure}

\subsection{Stiff nonlinear PDE: viscous Burgers' equation}
As a stiff nonlinear PDE problem, we study the viscous Burgers' equation,
\begin{equation}
\label{eq:Burgers_eqn}
u_t + u u_x=\nu u_{xx}+f \ \ \text{for} \ (x,t)\in (0,1) \times (0,1], \quad u=g \ \ \text{on} \ \{0,1\} \times (0,1]\;,
\end{equation}
with the true solution $u(x,t) = \cos(t)$, and the viscosity constant $\nu = 0.1$ (i.e., the main source of stiffness are the differential operators themselves). Here we choose a particularly simple manufactured solution to demonstrate that high weak stage order schemes do not fully remedy order reduction for nonlinear problems. However, we observe that these schemes still perform better than schemes with WSO $1$ in terms of accuracy and convergence order. This is important in the context of the demonstration in \cite{ketcheson2018dirk} that DIRK schemes with WSO up to 3 can exhibit clean and full order of convergence, even though problem \eqref{eq:Burgers_eqn} is nonlinear and thus outside the class of (linear) problems for which WSO is known to improve the accuracy of the LTE.

\begin{figure}[htb]
	\begin{minipage}[b]{.32\textwidth}
		\includegraphics[width=\textwidth]{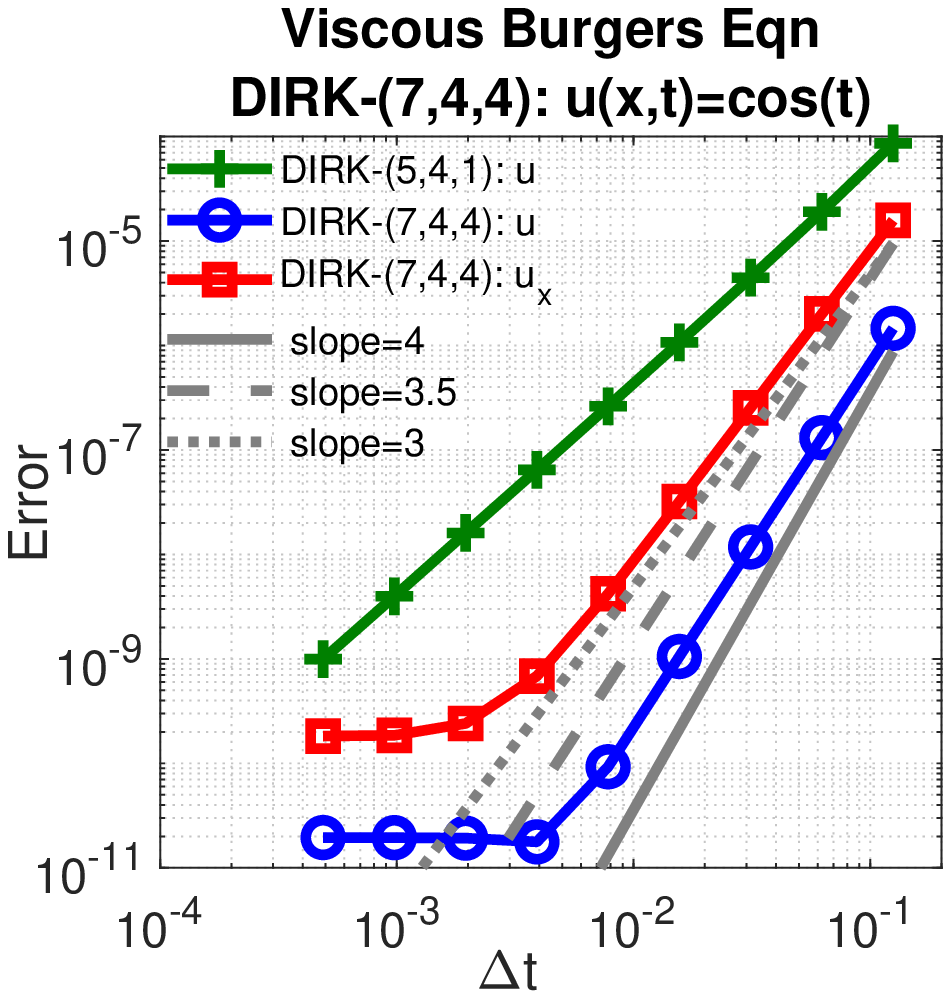}
	\end{minipage}
	\hfill
	\begin{minipage}[b]{.32\textwidth}
		\includegraphics[width=\textwidth]{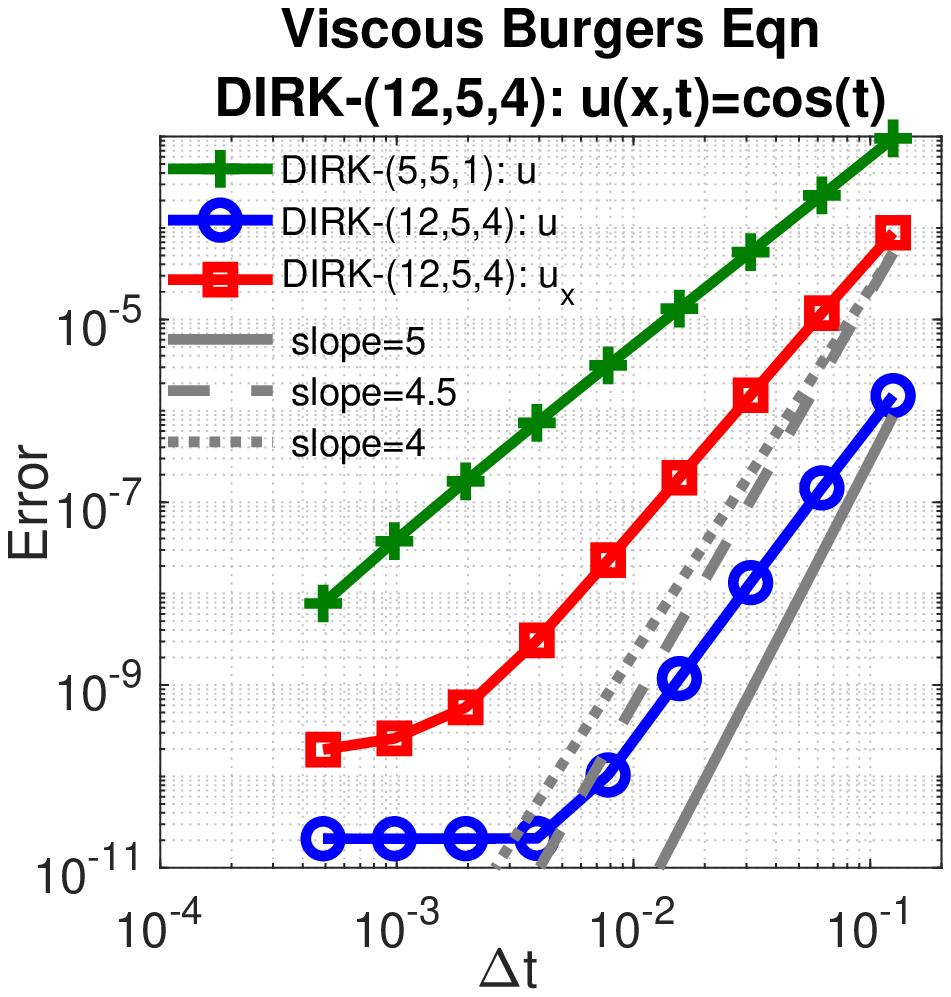}
	\end{minipage}
	\hfill
    \begin{minipage}[b]{.32\textwidth}
	    \includegraphics[width=\textwidth]{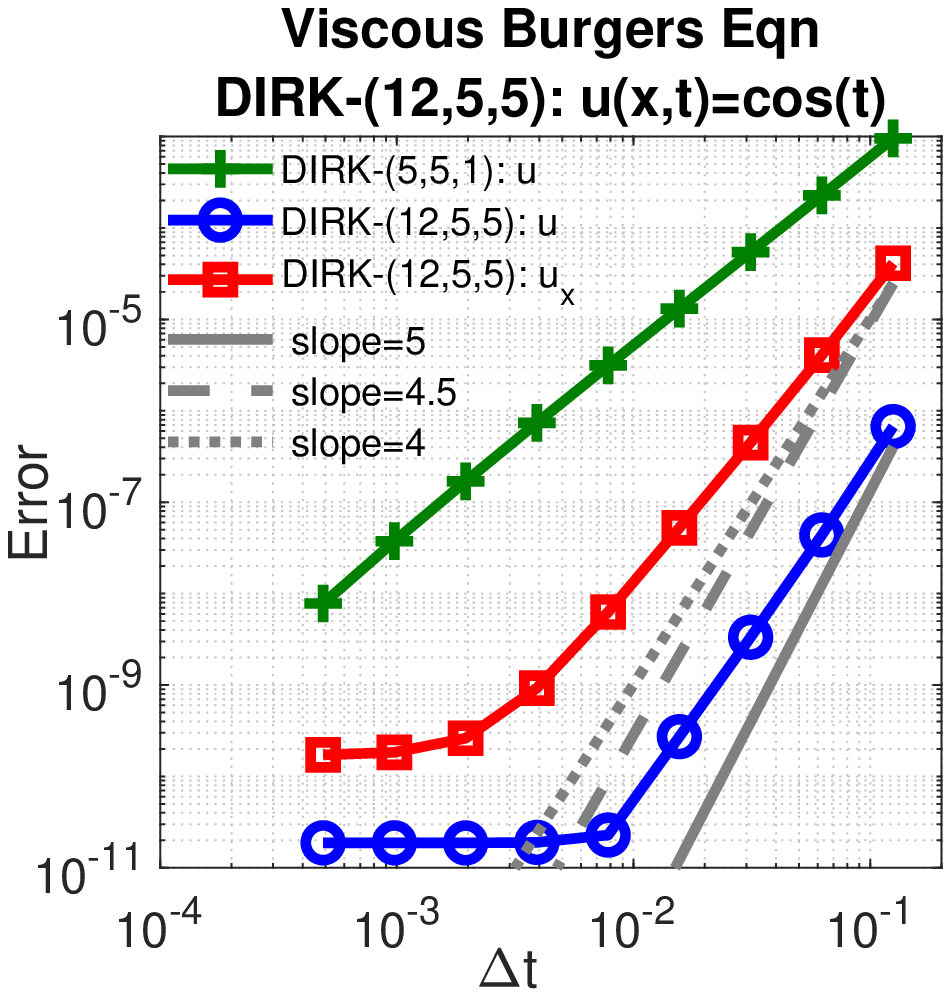}
    \end{minipage}
	\vspace{-.5em}
	\caption{Convergence for the viscous Burgers' equation using DIRK-$(7,4,4)$, DIRK-$(12,5,4)$, and DIRK-$(12,5,5)$.}
	\label{fig:VisBurgerEqn}
\end{figure}

Again, $6$th-order centered differences with $10^3$ cells are used to approximate the spatial derivatives, and the errors are evaluated at time $T=1$. Figure~\ref{fig:VisBurgerEqn} shows the convergence results for our new DIRK schemes with WSO $4$ and $5$ for the same test problem. We see that the high WSO schemes indeed turn out to generate a convergence order of $3$, which is better than what schems with WSO $1$ achieve, but there remains a reduction of order for the schemes of order above $3$ considered here.

\subsection{Stiff nonlinear ODE: Van der Pol oscillator}
To demonstrate that DIRK schemes with high weak stage order do not remedy order reduction for all types of problems, we consider, as a key benchmark example for stiff nonlinear ODE, the Van der Pol oscillator,
\begin{equation*}
    \frac{\mathrm{d}x}{\mathrm{d}t} = y \;, \quad 
    \frac{\mathrm{d}y}{\mathrm{d}t} = \mu \left(1-x^2\right)y-x\;,
\end{equation*}
with stiffness parameter $\mu = 500$, initial condition $(x(0),y(0))=(2,0)$, and final time $T=10$. For a range of time steps from $\Delta t= 0.5$ to $\Delta t \approx 2.44 \times 10^{-4}$, different DIRK schemes are applied, with Newton's method used to solve the nonlinear problems up to machine precision. The reference solution is calculated via the standard explicit RK$4$ method with time step $\Delta t = 10^{-6}$.
Figure~\ref{fig:VanDerPolEqn} shows the convergence for DIRK-$(7,4,4)$ and DIRK-$(12,5,5)$, clearly indicating that the high WSO in the DIRK schemes does not suffice to remove order reduction in the stiff regime ($10^{-2}<\Delta t<10^{-1}$).

\begin{figure}[htb]
	\begin{minipage}[b]{.47\textwidth}
	    \centering
		\includegraphics[width=.85\textwidth]{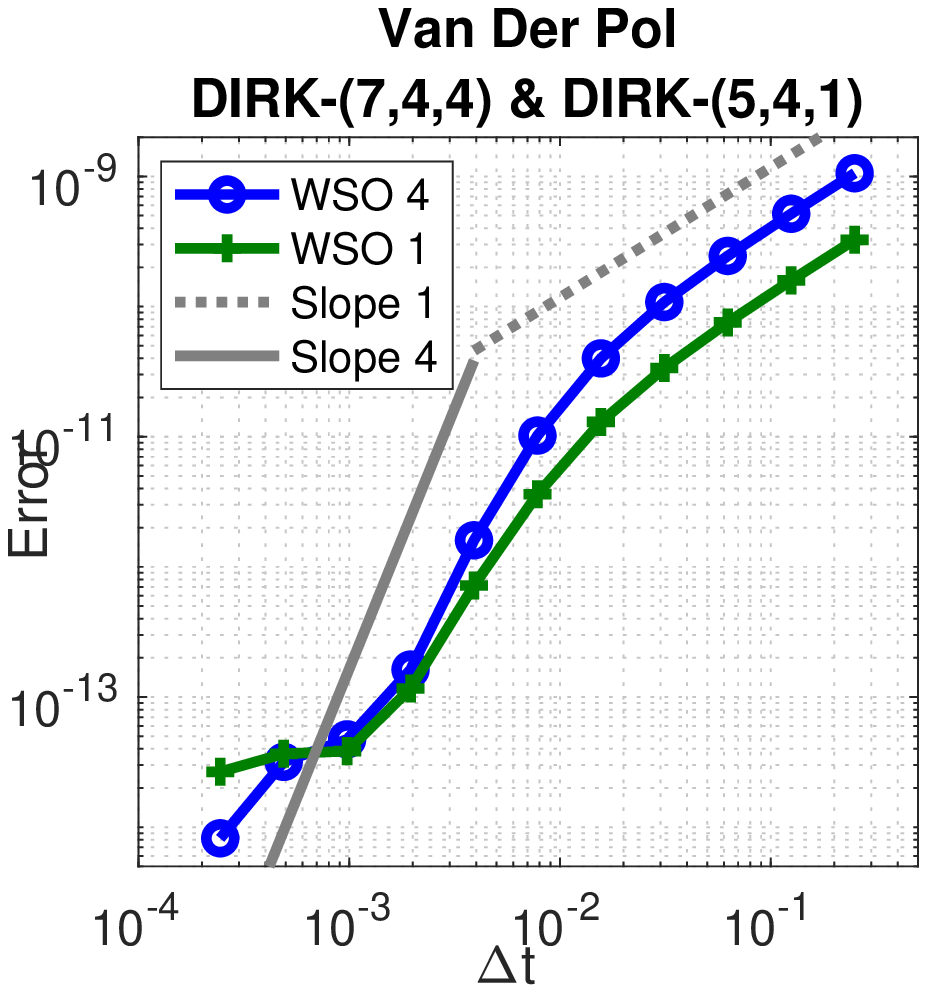}
	\end{minipage}
	\hfill
	\begin{minipage}[b]{.47\textwidth}
	    \centering
		\includegraphics[width=.85\textwidth]{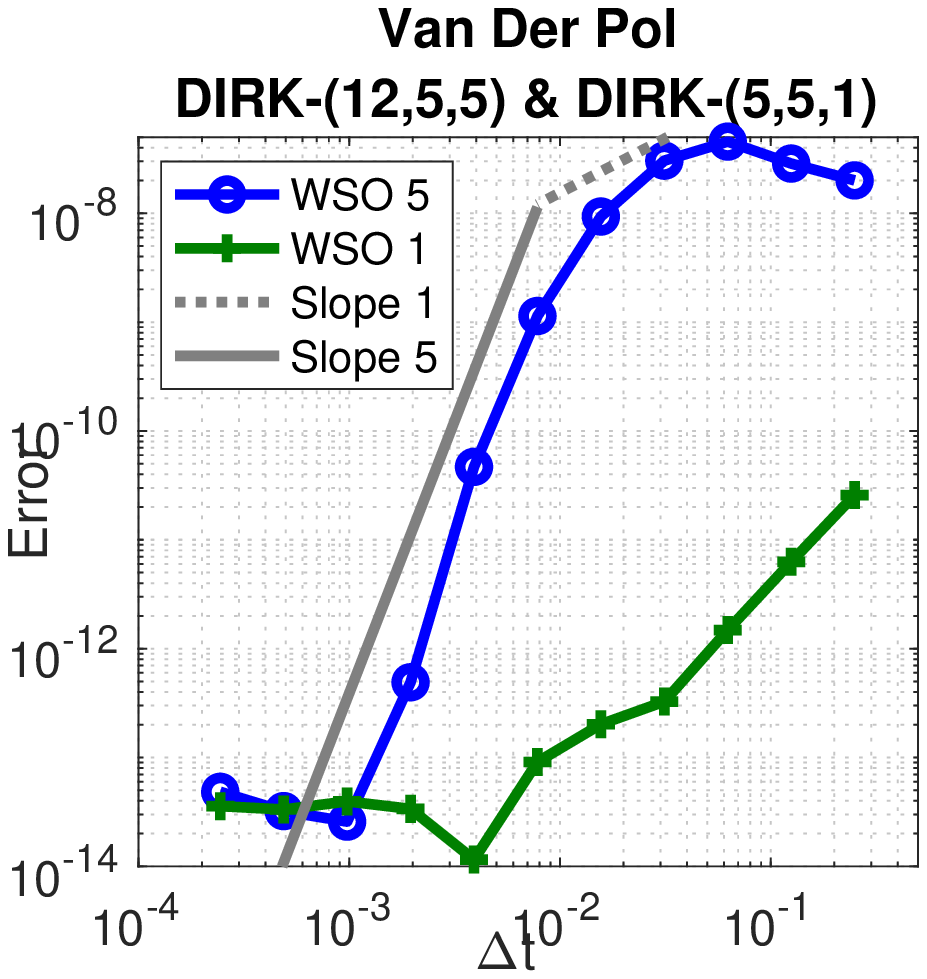}
	\end{minipage}
	\vspace{-.5em}
	\caption{Convergence for the Van der Pol oscillator using DIRK-$(7,4,4)$: $4$th order DIRK scheme with WSO $4$ (blue circles, left), and DIRK-$(12,5,5)$: $5$th order DIRK scheme with WSO $5$ (blue circles, right).}
	\label{fig:VanDerPolEqn}
\end{figure}

\section{Conclusions and Outlook}
\label{sec:DIRK_conclusions}
The results of this work can be seen as a reinforcement of the usefulness of weak stage order, which can remove order reduction in Runge-Kutta schemes applied to linear problems with time-independent operators. A key theoretical contribution of this paper is that it has been shown that WSO can indeed be extended beyond WSO 3 (which is important because a special case of WSO had previously been shown to be limited to WSO 3 \cite{ketcheson2018dirk}). Moreover, utilizing a general theory of WSO \cite{BiswasKetchesonSeiboldShirokoff2022}, three concrete new DIRK schemes, DIRK-$(7,4,4)$, DIRK-$(12,5,4)$, and DIRK-$(12,5,5)$, have been constructed with high WSO and other desirable properties: stiff accuracy, L-stability, and optimized error constants. These new schemes have been demonstrated to be practically useful, as they successfully address the order reduction problem in a variety of test problems, both those covered by the theory (linear problems with time-independent operators), as well as some (but not all) problems beyond the scope of the theory.

Because the new schemes have the standard form of RK methods, they can be easily incorporated into existing software and thus may be immediately useful for practitioners who seek to remedy order reduction while using DIRK time stepping.

The results presented here give rise to several further questions and research directions. First, problem $\textcolor{black}{(M)}$ in \S\ref{sec:optProblemDIRK} that characterizes optimal DIRK schemes is a polynomial optimization problem. While successfully solved via generic approaches herein, tailored modern optimization methods that yield provably globally optimal solutions represent a natural next step. Second, the number of stages used by the schemes provided here ($7$ and $12$ stages, respectively) is larger
than the theoretical minimum
number of stages implied by the theorems we provide. Both sharp bounds and concrete DIRK schemes that realize the minimum number of stages remain to be found. Third, specific explorations of WSO for EDIRKs ($a_{11} = 0$), SDIRKs (all $a_{ii}$ identical), and also explicit RK schemes, remain practically relevant open directions of research.

\appendix
\section{List of New DIRK Schemes} 
\label{app:butcherTableau}
See Table~\ref{table:DIRK_schemes} for the new DIRK schemes with high weak stage order.

\begin{table}
\centering
\caption{Butcher's tableau for DIRK-$(7,4,4)$ (left), DIRK-$(12,5,4)$ (middle), and DIRK-$(12,5,5)$ (right).}
\label{table:DIRK_schemes}
\begin{adjustbox}{angle=90}
\resizebox{20cm}{!}{%
	\begin{tabular}{l|l l l l l l l}
     1.290066345260422e-01  &    1.290066345260422e-01 \\                        
     4.492833135308985e-01  &    3.315354455306989e-01  &   1.177478680001996e-01  \\                      
     9.919659086525534e-03  &   -8.009819642882672e-02  &  -2.408450965101765e-03  &   9.242630648045402e-02  \\                       
     1.230475897454758e+00  &   -1.730636616639455e+00  &   1.513225984674677e+00  &   1.221258626309848e+00  &   2.266279031096887e-01  \\                        
     2.978701803613543e+00  &    1.475353790517696e-01  &   3.618481772236499e-01  &  -5.603544220240282e-01  &   2.455453653222619e+00  &   5.742190161395324e-01  \\                            
     1.247908693583052e+00  &    2.099717815888321e-01  &   7.120237463672882e-01  &  -2.012023940726332e-02  &  -1.913828539529156e-02  &  -5.556044541810300e-03  &   3.707277349712966e-01 \\
     1.000000000000000e+00   &   2.387938238483883e-01  &   4.762495400483653e-01  &   1.233935151213300e-02  &   6.011995982693821e-02  &   6.553618225489034e-05  & -1.270730910442124e-01            &  3.395048796261326e-01 \\ \hline
    &   2.387938238483883e-01  &   4.762495400483653e-01  &   1.233935151213300e-02  &   6.011995982693821e-02  &   6.553618225489034e-05  &  -1.270730910442124e-01            &  3.395048796261326e-01
			 				   			 
	\end{tabular}%
	}%
\end{adjustbox}
\hspace{2em}
\begin{adjustbox}{angle=90}
\resizebox{20cm}{!}{%
		\begin{tabular}{l|l l l l l l l l l l l l l}
         2.345371908646273e-01   &    2.345371908646273e-01  \\
         7.425871511958302e-01   &    6.874344413888787e-01  &   5.515270980695153e-02   \\
         3.296674204078279e-02   &   -1.183552669539587e-01  &   5.463563002913454e-03  &   1.458584459918280e-01 \\
         7.379564717201322e-01   &   -1.832235204042292e-01  &   5.269029412008775e-02  &   8.203685085133529e-01  &   4.812118949092085e-02 \\
         2.376643917109970e-01   &    9.941572060659400e-02  &   4.977904930055774e-03  &   5.414758174284321e-02  &  -1.666571741820749e-03  &     8.078975617332473e-02  \\
         1.750238160341377e+00   &   -9.896614721582678e-01  &   2.860682690577833e+00  & -1.236119341063179e+00   &   2.130219523351530e+00  &    -1.260655031676537e+00  &   2.457717913099987e-01  \\
         2.990308150015702e+00   &   -5.656238413439102e-02  &   1.661985685769353e-01  &   6.464600922362508e-01  &   6.608854962269927e-01  &   3.736054198873429e-01     & 6.294456964407685e-01   &     5.702752607818027e-01  \\ 
         2.882138003112822e+00   &   8.048962104724392e-01  &  -6.232034990249100e-02  &   5.737234603323347e-01   &  -9.613723511489970e-02  &    5.524106361737929e-01  &     5.961002486833255e-01  &   1.978411600659203e-01  &   3.156238724024008e-01 \\
         2.914399924907188e+00   &    -1.606381759216300e-01 &  6.833397073337708e-01  &   4.734578665308685e-01   &   8.037708984872738e-01  &   -1.094498069459834e-02  &     6.151263362711297e-01  &     3.908946848682723e-01 &  8.966103265353116e-02  &   2.973255537857041e-02 \\
         2.573507348677332e+00   &     7.074283235644631e-01 &  4.392037300952482e-01  &  -3.623592480237268e-02   &   7.189990308645932e-04  &   5.820968279166545e-01  &     3.302003177175218e-01   &    -2.394564021215881e-01 & -7.540283547997615e-03  &   1.702137469523672e-01   &   6.268780138721711e-01   \\
         3.567266961364713e+00   &     1.361197981133694e-01 & -7.486549901902831e-01  &   1.893908350024949e+00   &   3.940485196730028e-01  &   6.240233526545023e-02  &    7.511983862200027e-01   &    -5.283465265730526e-01 & -1.661625677872943e+00  &   9.998723833190827e-01   &   1.377776742457387e+00  &   8.905676409277480e-01 \\
         1.000000000000000e+00	 &   -7.433675378768276e-01 &  1.490594423766965e-01  &  -2.042884056742363e-02   &   8.565329438087443e-04  &   1.357261590983184e+00  &     2.067512027776675e-03   &     9.836884265759428e-02 & -1.357936974507222e-02  &  -5.428992174996300e-02   &  -3.803299038293005e-02  &  -9.150525836295019e-03  & 2.712352651694511e-01 \\ \hline
         &   -7.433675378768276e-01 &  1.490594423766965e-01  &  -2.042884056742363e-02   &   8.565329438087443e-04  &   1.357261590983184e+00  &     2.067512027776675e-03   &     9.836884265759428e-02 & -1.357936974507222e-02  &  -5.428992174996300e-02   &  -3.803299038293005e-02  &  -9.150525836295019e-03  & 2.712352651694511e-01
		\end{tabular}%
	}
\end{adjustbox}
\hspace{2em}
\begin{adjustbox}{angle=90}
\resizebox{20cm}{!}{%
		\begin{tabular}{l|l l l l l l l l l l l l l}
             4.113473525867655e-02  &     4.113473525867655e-02  \\
             2.269850660400232e-01  &     1.603459327727949e-01  &     6.663913326722831e-02 \\
             6.222969192243949e-01  &    -3.424389044264752e-01  &     8.658006324816373e-01  &     9.893519116923277e-02 \\
             1.377989449231234e+00  &     9.437182028870806e+00  &    -1.088783359642350e+01  &     2.644025436733866e+00  &     1.846155800500574e-01 \\
             1.259841986970257e+00  &    -3.425409029430815e-01  &     5.172239272544332e-01  &     9.163589909678043e-01  &     5.225142808845742e-02  &     1.165485436026433e-01  \\
             1.228350442796143e+00  &    -2.094441177460360e+00  &     2.577655753533404e+00  &     5.704405293326313e-01  &     1.213637180023516e-01  &    -4.752289775376601e-01  &    5.285605969257756e-01 \\
             1.269855051265635e+00  &    3.391631788320480e-01  &   -2.797427027028997e-01   &    1.039483063369094e+00 &  5.978770926212172e-02   &   -2.132900327070380e-01   & 8.344318363436753e-02  &     2.410106515779412e-01 \\
             2.496200652601413e+00  &     5.904282488642163e+00  &    3.171195765985073e+00  &    -1.236822836316587e+01  &    -4.989519066913001e-01  &     2.160529620826442e+00    &   1.916104322021480e+00  &     1.988059486291180e+00  &    2.232092386922440e-01 \\
             2.783820705331141e+00  &     4.616443509508975e-01  &   -1.933433560549238e-01   &   -1.212541486279519e-01   &    6.662362039716674e-02  &  4.254912950625259e-01  &  7.856131647013712e-01  &   8.369551389357689e-01  &  1.604780447895926e-01  &    3.616125951766939e-01  \\
             3.337101417632813e+00  &    -7.087669749878204e-01   &   6.466527094491541e-01  &   4.758821526542215e-01   &  -2.570518451375722e-01   &  1.123185062554392e+00  &  5.546921612875290e-01   &  3.192424333237050e-01  &  3.612077612576969e-01  &  5.866779836068974e-01  &  2.353799736246102e-01  \\
             4.173423133876636e+00  &     4.264162484855930e-01   &  1.322816663477840e+00  &  4.245673729758231e-01  &  -2.530402764527700e+00  &  -7.822016897497742e-02  &     1.054463080605071e+00 &  4.645590541391895e-01  &    1.145097379521439e+00  &  4.301337846893282e-01  &   1.499513057076809e+00  &  1.447942640822165e-02   \\
             1.000000000000000e+00  &   1.207394392845339e-02  &  5.187080074649261e-01  &  1.121304244847239e-01 &  -4.959806334780896e-03  &  -1.345031364651444e+00  &     3.398828703760807e-01  &  8.159251531671077e-01  &  -2.640104266439604e-03  &  1.439060901763520e-02  &  -6.556567796749947e-03  &  6.548135446843367e-04  &     5.454220210658036e-01  \\ \hline
             &   1.207394392845339e-02  &  5.187080074649261e-01  &  1.121304244847239e-01 &  -4.959806334780896e-03  &  -1.345031364651444e+00  &     3.398828703760807e-01  &  8.159251531671077e-01  &  -2.640104266439604e-03  &  1.439060901763520e-02  &  -6.556567796749947e-03  &  6.548135446843367e-04  &     5.454220210658036e-01
		\end{tabular}%
	}
\end{adjustbox}
\end{table}

\section*{Acknowledgments}
This material is based upon work supported by the National Science Foundation under Grants No.\ DMS--2012271 (Biswas, Seibold), No.\ DMS--1952878 (Seibold), and No.\ DMS--2012268 (Shirokoff).


\vspace{1.5em}
\bibliographystyle{plain}
\bibliography{references}

\vspace{1.5em}
\end{document}